\documentclass[10pt,reqno,final]{amsart}
\usepackage[notcite,notref]{showkeys}
\usepackage{url,hyperref,multirow}
\usepackage{color}
\usepackage{subfigure}
\usepackage{epsfig,amssymb,amsmath,version,slashbox}
\usepackage{amssymb,version,graphicx,fancybox,mathrsfs,pifont,booktabs}
\catcode`\@=11 \theoremstyle{plain}
\@addtoreset{equation}{section}   
\renewcommand{\theequation}{\arabic{section}.\arabic{equation}}
\@addtoreset{figure}{section}
\renewcommand\thefigure{\thesection.\@arabic\c@figure}
\@addtoreset{table}{section}
\renewcommand\thetable{\thesection.\@arabic\c@table}

\newtheorem{thm}{\bf Theorem}[section]
\newtheorem{cor}{\bf Corollary}[section]
\newtheorem{prop}{Proposition}[section]

\newenvironment{corollary}{\begin{cor}} {\end{cor}}
\newtheorem{lmm}{\bf Lemma}[section]

\newenvironment{lemma}{\begin{lmm}}{\end{lmm}}
\theoremstyle{remark}
\newtheorem{rem}{Remark}[section]

\newcommand{\re}[1]{(\ref{#1})}

\def \ri {{\rm i}}

\def \af {\alpha}
\def \bt {\beta}

\begin{document}
\baselineskip 13pt
\bibliographystyle{plain}
\title[Jacobi expansions and Gegenbauer-Gauss quadrature of analytic functions]
{Sharp error bounds for Jacobi  expansions and Gengenbauer-Gauss  quadrature of analytic functions}
\author[X. Zhao, ~ L. Wang ~  $\&$~  Z. Xie] { Xiaodan Zhao${}^{1},$ \quad Li-Lian Wang${}^{1}$ \quad
and \quad
  Ziqing Xie${}^{2}$}
\thanks{\noindent${}^{1}$ Division of Mathematical Sciences, School of Physical
and Mathematical Sciences,  Nanyang Technological University,
637371, Singapore. The research of the authors is partially
supported by Singapore AcRF Tier 1 Grant RG58/08.\\
\indent${}^{2}$ School of Mathematics and Computer Science, Guizhou
Normal University, Guiyang, Guizhou 550001, China; College of
Mathematics and Computer Science, Hunan Normal University, Changsha,
Hunan 410081, China. The research of the author is partially
supported by the NSFC (10871066) and the Science and Technology
Grant  of Guizhou Province  (LKS[2010]05).
 }

\begin{abstract} This paper provides a rigorous and delicate  analysis for  exponential decay of Jacobi polynomial expansions of
analytic functions associated with the Bernstein ellipse.  Using an argument that
 can recover the best estimate for the Chebyshev expansion, we derive various new and sharp bounds of the expansion coefficients, which are  featured with explicit dependence of
 all related parameters and valid for  degree $n\ge 1$.  We demonstrate the sharpness of the estimates by comparing with existing ones, in particular, the very recent results in \cite[SIAM J. Numer. Anal., 2012]{xiang2012error}.
We also extend this argument to estimate the Gegenbauer-Gauss quadrature remainder of analytic functions,
which  leads to some  new tight bounds for quadrature errors. 
\end{abstract}

\maketitle

\thispagestyle{empty}

\vspace*{-12pt}
\section{Introduction}

The spectral method employs global orthogonal polynomials or Fourier
complex exponentials as basis functions,  so it enjoys high-order
accuracy  (with only a few basis functions), if the underlying
function is smooth (and periodic in the Fourier case). The
convergence rate $O(n^{-r})$, where $n$ is the number of basis
functions involved in a spectral expansion and $r$ is related to the
Sobolev-regularity of the underlying function, is typically
documented in various monographs on spectral methods
\cite{gottlieb1977numerical,Funa92,Fornberg96,BernardiMaday97,Guo98,Tref00,Boyd01,CHQZ06,GotHes07,ShenTangWang2011}.
It is also widely appreciated that if the function under consideration is  analytic, the  convergence rate is of exponential order  $O(q^n)$ (for constant $0<q<1$). However, there appears very limited discussions of such error bounds
 (mostly mentioned, but not proved) in \cite{Fornberg96,Tref00,Boyd01}. Indeed, as commented by Hale and Trefethen \cite{Hale.Tr08},
 the general idea of such convergence goes back to Bernstein in early nineties, but such results do not appear in many textbooks or
 monographs, and there is not much uniformity in the constants in the upper bounds.


 An important result in Bernstein \cite{Bernstain} (1912) (also see \cite{lorentz})
 states that $u$ is analytic on $[-1, 1],$ if and only if
$$\sup\lim_{N\to \infty} \sqrt[N]{E_N(u)}=\frac 1 \rho,\quad E_N(u)=\inf_{v\in P_N}\|v-u\|_{\infty},   $$
where $P_N$ is the polynomial space of degree no more than $N$, and
$\rho> 1$ is the sum of the semi-axes of the maximum  ellipse
${\mathcal E}_\rho$  with foci $\pm 1,$ known as the {\em Bernstein
ellipse}, on and within which $u$ can be analytically extended to.
One immediate implication is that the best polynomial approximation
in the maximum norm  enjoys exponential convergence.  A more precise estimate for the Chebyshev
expansion can be found in various approximation theory texts (see
e.g., \cite[Theorem 3.8]{rivlin1990chebyshev} and  \cite[Theorem
5.16]{Mason03}):
\begin{equation}\label{chebmars}
|\hat u_n^C|\le \frac{2M}{\rho^n},\;\;  \forall n\ge 0; \quad
\big\|u-S_N^Cu\big\|_{\infty} \le \frac{M}{(\rho-1) \rho^N},
\end{equation}
where $M=\max_{z\in {\mathcal E}_\rho}|u(z)|,$ $\{\hat u_n^C\}$ are
Chebyshev expansion coefficients of $u,$ and $S_N^C u$ is the
partial sum involving the first $N+1$ terms. One  also refers to
\cite{szeg75,davis1975interpolation,rivlin1990chebyshev,Boyd94,Mason03,TrefSIAMRev08,Xiang2011}
and the references therein for verification/description of
exponential convergence of Fourier, Chebyshev or Legendre
expansions. We remark that Gottlieb and Shu  et al
\cite{gottlieb1992gibbs,Got.S97}
 studied   exponential convergence of  Gegenbauer expansions (when the parameter grows
linearly with $n$) in the context of  defeating the Gibbs phenomenon.

Here, we particularly  highlight that a very recent paper of  Xiang
\cite{xiang2012error} provided a  simple approach to obtain
the bounds for Jacobi expansion coefficients  of analytic functions  on and
within the Bernstein ellipse ${\mathcal E}_\rho:$
\begin{equation}\label{xiangres}
|\hat u_n^{\alpha,\beta}|\le \frac{2M}{\rho^{n-1}(\rho -1)}\sqrt{\frac{\gamma_0^{\af,\bt}}{\gamma_n^{\af,\bt}}}\;\; {\rm where}\;\;  \hat u_n^{\alpha,\beta}=\frac{1}{\gamma_n^{\af,\bt}}\int_{-1}^1 u(x) J_n^{\af,\bt}(x)\omega^{\af,\bt}(x) dx.
\end{equation}
Here, $\{J_n^{\af,\bt}\} (\af,\bt>-1)$ are Jacobi polynomials mutually orthogonal with the weight function
 $\omega^{\af,\bt}(x)=(1-x)^\af(1+x)^\bt$ and with the normalization factor $\gamma_n^{\af,\bt}$ (cf. \eqref{gammafd}). The key step is to insert the Chebyshev expansion $u(x)=\sum_{j=0}^\infty \hat u_j^C T_j(x)$ into the Jacobi expansion coefficients and   rewrite
 $$
 \hat u_n^{\alpha,\beta}= \frac{1}{\gamma_n^{\af,\bt}}\sum_{j=0}^\infty \hat u_j^C \int_{-1}^1 T_j(x) J_n^{\af,\bt}(x)\omega^{\af,\bt}(x) dx,
 $$
so the  bound for the Chebyshev coefficient in \eqref{chebmars} could be used. 

The first purpose of the paper is to take a different approach to derive sharp estimates for general Jacobi
  expansion of analytic functions. The assertion of sharpness is in the following sense:
 \begin{itemize}
\item[(i)] The bound for general Jacobi case is tighter than   \eqref{xiangres} (see Remark \ref{tightness}).
\item[(ii)]  Refined estimates can be obtained for Gegenbauer expansion ($\af=\bt>-1$),
  Chebyshev-type expansion ($\af=k-1/2, \bt=l-1/2$ for non-negative integers $k,l$), and Legendre-type  expansion
  ($\af=k, \bt=l$ for non-negative integers $k,l$). The argument can recover the  bounds known to be the sharpest (e.g., the Chebyshev case), and  some obtained estimates are new and significantly improve the existing ones (see e.g., Remark \ref{legenrem}).
\end{itemize}

 A second purpose of this work is to extend the argument to analyze Gegenbauer-Gauss quadrature of analytical functions.
Recall that the remainder of  Gauss-quadrature with the nodes and weights $\{x_j,\omega_j\}_{j=1}^n,$ takes the form (see e.g., \cite{Dav.R84}):
 \begin{equation}\label{gausslob}
 E_n[u]=\int_{-1}^1 u(x)\omega(x)\,dx-\sum_{j=1}^{n}u(x_j)\omega_j=\frac{1}{\pi \ri}\oint_{\mathcal E_\rho} \frac{q_n(z)}{p_n(z)}u(z)\, dz,
 \end{equation}
 where $\{x_j\}_{j=1}^n$ are the zeros of $p_n(x),$ orthogonal with respect to the weight function $\omega(x),$
 and
   \begin{equation}\label{gausslob2}
 q_n(z)=\frac 1 2 \int_{-1}^1 \frac{p_n(x)\omega(x)}{z-x}\, dx.
  \end{equation}
The estimate of quadrature errors has attracted much attention (see e.g.,
\cite{Chawla68,Chawla68SIAM,Basu70,Gautschi83,Dav.R84,GTV90,Hunter1995,Hunter98}). Among these results, intensive discussions have been centered around the Chebyshev case and its family, e.g., Chebyshev of the second kind,  but with very limited results  even for  Legendre-Gauss quadrature (see e.g., \cite{Chawla69,Kambo70}). In fact, the analysis heavily relies on the availability of explicit expression of $p_n(z)$ on $\mathcal E_\rho.$ Armed with a delicate estimate of $q_n(z)$ (in the first part of the paper) and the explicit formula of Gegenbauer polynomial
in \cite{XieWangZhao2011}, we are able to derive sharp bound for the Gegenbauer-Gauss quadrature errors.


We remark that there has been much interest in estimating spectral differentiation errors of analytic functions.
Tadmor \cite{tadmor1986exponential}  first attempted to estimate the aliasing errors to verify exponential convergence of Fourier and Chebyshev spectral differentiation with a different assumption on analyticity. The results for analyticity characterized by the Bernstein ellipse  include  Reddy and Weideman \cite{ReddyWeideman2005} for Chebyshev case,
and Xie, Wang and Zhao \cite{XieWangZhao2011} for Gegenbauer spectral differentiation.
It is also interesting to point out that Zhang \cite{ZhangZM04,ZhangZM08,ZhangInp2012} studied superconvergence of spectral interpolation and differentiation.   We stress that the analysis apparatuses  and arguments in this pipeline  are  different from these in this work.

 The rest of this paper is organized as follows. In Section \ref{sect2},
we provide  sharp  bounds for general Jacobi expansions of analytic functions, followed by some
 refined results for Chebyshev-type and Legendre-type expansions.
In Section \ref{GGquadratrue}, we extend the argument to analyze Gegenbauer-Gauss quadrature errors.
In the final section, we provide results to show the sharpness of the bounds by
comparing them with existing ones.

\section{Sharp bounds for Jacobi expansions}\label{sect2}

We derive in this section  sharp bounds for Jacobi  expansions of  functions analytic
on and within the Bernstein ellipse ${\mathcal E}_\rho$.

\subsection{Preliminaries}
It is known (see e.g., \cite{davis1975interpolation}) that the Bernstein ellipse is transformed  from the circle
\begin{equation}\label{crho}
{\mathcal C}_\rho=\big\{w=\rho e^{\ri\theta} : \theta\in [0,2\pi]\big\},\quad \rho>1,
\end{equation}
via the conformal mapping: $z=(w+w^{-1})/2,$ namely,
\begin{equation}\label{Berellips}
{\mathcal E}_\rho:=\Big\{z\in {\mathbb C}~:~ z=\frac 1
2(w+w^{-1})\;\; \text{with}\;\; w=\rho e^{\ri \theta},\; \theta\in
[0,2\pi] \Big\},
\end{equation}
where ${\mathbb C}$ is  the set of all complex numbers, and
$\ri=\sqrt{-1}$ is the complex unit.  It  has the foci at
 $\pm 1,$ and the  major and minor semi-axes are
 \begin{equation}\label{semiaxis}
 a=\frac 1 2 \big(\rho+\rho^{-1}\big),\quad  b=\frac 1 2 (\rho-\rho^{-1}),
 \end{equation}
respectively, so the sum of  two semi-axes is $\rho.$   The perimeter of ${\mathcal E}_\rho$ has the bound
\begin{equation}\label{Lehp}
L({\mathcal E}_\rho) \leq\pi\sqrt{\rho^2+\rho^{-2}},
\end{equation}
which overestimates the perimeter by less than $12$ percent (cf. \cite{ReddyWeideman2005}). The
distance from ${\mathcal E}_\rho$ to the interval $[-1,1]$ is
\begin{equation}\label{drho}
d_\rho=\frac12(\rho+\rho^{-1})-1.
\end{equation}
We see that $d_\rho$ increases  with respect to $\rho,$ and
$d_\rho\to 0^+$ as $\rho\to 1^+$ (so the ellipse reduces to the
interval $[-1,1]$). Thus, by the theory of analytic
continuation, we have that  {\em for any analytic function $u$ on $[-1,1],$
there always exists a Bernstein  ellipse ${\mathcal E}_\rho$ with
$\rho>1$ such that the continuation  of $u$ is analytic on and
within ${\mathcal E}_\rho.$}
 Hereafter, we denote by
 \begin{equation}\label{analyA}
{\mathcal A}_\rho:=\big\{u\,:\, u\;\;\text{is analytic on and
within}\;\; {\mathcal E}_\rho\big\}, \quad 1<\rho<\rho_{\rm max},
\end{equation}
where ${\mathcal E}_{\rho_{\rm max}}$ labels the largest ellipse
within which $u$ is analytic. In particular, if $\rho_{\rm
max}=\infty,$  $u$ is an entire function.

Throughout this paper, the Jacobi polynomials, denoted by $J_n^{\af,\bt}(x)$ (with $\af,\bt>-1$ and $x\in I:=(-1,1)$), are
 normalized as in Szeg\"o \cite{szeg75}, i.e.,
\begin{equation}\label{jacobi_orth}
\int_{-1}^1J_n^{\alpha,\beta}(x)J_m^{\alpha,\beta}(x)
\omega^{\alpha,\beta}(x)\, dx =\gamma_n^{\af,\bt} \delta_{m,n},
\end{equation}
where $\omega^{\af,\bt}(x)=(1-x)^\af(1+x)^\bt,$  $\delta_{m,n}$ is
the Kronecker delta, and
\begin{equation}\label{gammafd}
\gamma_n^{\af,\bt}=\frac{2^{\alpha+\beta+1}\Gamma(n+\alpha+1)\Gamma(n+\beta+1)}{(2n+\alpha+\beta+1)n!\Gamma(n+\alpha+\beta+1)}.
\end{equation}
In Appendix \ref{sect:Jacobi},  we  collect the relevant properties of Jacobi polynomials.

In the analysis, we also use the following property of the Gamma function, derived from
\cite[Eq. (6.1.38)]{Abr.I64}:
\begin{equation}\label{stirlingfor}
\Gamma(x+1)=\sqrt{2\pi}
x^{x+1/2}\exp\Big(-x+\frac{\theta}{12x}\Big),\quad \forall\,
x>0,\;\;0<\theta<1.
\end{equation}

\begin{lemma}\label{lemGammaratio} For any constants  $a, b,$
 we have that  for  $n\ge1, $  $n+a>1$ and $n+b>1,$
\begin{equation}\label{Gammaratio}
\frac{\Gamma(n+a)}{\Gamma(n+b)}\le \Upsilon_n^{a,b} n^{a-b},
\end{equation}
where
\begin{equation}\label{ConstUpsilon}
\Upsilon_n^{a,b}=\exp\Big(\frac{a-b}{2(n+b-1)}+\frac{1}{12(n+a-1)}+\frac{(a-b)^2}{n}\Big).
\end{equation}
\end{lemma}
\begin{proof}
Let $\theta_1, \theta_2$ be two constants in $(0,1).$ We find
from  \eqref{stirlingfor} that
\[
\begin{split}
\frac{\Gamma(n+a)}{\Gamma(n+b)}&=\frac{(n+a-1)^{n+a-1/2}}{(n+b-1)^{n+b-1/2}}\exp\Big(-a+b+\frac{\theta_1}{12(n+a-1)}-\frac{\theta_2}{12(n+b-1)}\Big)
\\&\le
(n+a-1)^{a-b}\Big(1+\frac{a-b}{n+b-1}\Big)^{n+b-1/2}\exp\Big(-a+b+\frac{1}{12(n+a-1)}\Big)
\\&\le n^{a-b}\Big(1+\frac{a-b}{n}\Big)^{a-b}\exp\Big(-a+b+\frac{(a-b)(n+b-1/2)}{n+b-1}+\frac{1}{12(n+a-1)}\Big)
\\& \le
n^{a-b}\exp\Big(\frac{a-b}{2(n+b-1)}+\frac{1}{12(n+a-1)}+\frac{(a-b)^2}{n}\Big)
:=\Upsilon_n^{a,b} n^{a-b},
\end{split}
\]
where we used the fact that $1+x\le e^x,$ for real $x.$
\end{proof}

\begin{rem}\label{gammarem} Applying \eqref{ConstUpsilon} to  $\gamma_n^{\af,\bt}$ leads to that for $\af,\bt>-1$, $n\ge
1$ and $n+\af+\bt>0$,
\begin{equation}\label{gammacoef}
\gamma_n^{\af,\bt}\le \frac{2^{\af+\bt+1}}{2n+\af+\bt+1}\Upsilon_n^{\af+1,1}\Upsilon^{\bt+1,\af+\bt+1}_n.
\end{equation}
Note that for fixed $a$ and $b,$
\begin{equation}\label{consUps}
\Upsilon_n^{a,b}=1+O(n^{-1}),
\end{equation}
as it behaves like $e^{1/n}.$  \qed
\end{rem}

\vskip 6pt
\subsection{Main tools}~~ 

\vskip 5pt
Our starting point is the following  important representation.
\begin{lemma}\label{LM1} Let  $\{\hat u_n^{\alpha,\beta}\}$ be the Jacobi polynomial expansion
coefficients given by
\begin{equation}\label{jacobicoe}
\hat{u}_n^{\af,\bt}=\frac{1}{\gamma_n^{\af,\bt}}\int_{-1}^1
u(x)J_n^{\af,\bt}(x)\omega^{\af,\bt}(x)\, dx,\quad \af,\bt>-1,\;\; n\ge 0.
\end{equation}
If  $u\in {\mathcal A}_\rho$ with $\rho>1,$   we have the representation:
\begin{equation}\label{expancoef}
\hat u_n^{\af,\bt}=\frac{1}{\pi
\ri}\sum_{j=0}^{\infty}\sigma_{n,j}^{\af,\bt}\oint_{\mathcal{E}_\rho}\frac{u(z)}{w^{n+j+1}}\,dz,
\quad n\ge 0,
\end{equation}
where  $z=(w+w^{-1})/2$ with $w=\rho e^{\ri \theta},\; \theta\in [0,2\pi],$ and
\begin{equation}\label{sigmanj}
\sigma_{n,j}^{\af,\bt}= \frac{1}{\gamma_n^{\af,\bt}}
 \int_{-1}^1U_{n+j}(x) J_n^{\alpha,\beta}(x)\omega^{\alpha,\beta}(x)
 \,dx, \quad n, j\ge 0.
 \end{equation}
Here,  $U_{n+j}(x)$ is the Chebyshev polynomial of the second kind of
degree $n+j$ {\rm(cf. \eqref{chbytype2})}.
\end{lemma}
Actually, the formula \eqref{expancoef}-\eqref{sigmanj} can  be obtained by assembling several formulas in Szeg\"o \cite{szeg75}, and then using  the generating function of $U_{k}(x)$ (cf. \cite{Abr.I64}).  For the readers' reference, we sketch its derivation in Appendix \ref{pflemma3.1}.
\vskip 4pt

The establishment of sharp bounds heavily  relies on estimating
$\sigma_{n,j}^{\af,\bt}.$ The following explicit formulas  follow
from  \eqref{sigmanj} and some  properties of Jacobi polynomials
listed in Appendix \ref{sect:Jacobi}. We remark that the formula
\eqref{constcase2} can be found in various books e.g.,
\cite{davis1975interpolation,Mason03}, while the formula
\eqref{constcase3}  is due to Heine (see \cite{Chawla68}). We also
highlight that the formula \eqref{tempwa0} for the general Jacobi
case  seems new.
\begin{corollary}\label{sigmanjcor} Let $n\ge 0.$
\begin{itemize}
\item[(i)] For $\alpha=\beta>-1$ {\rm(}ultraspherical/Gegenbauer polynomial\,{\rm)\footnote{In this paper, we do not distinguish between ultraspherical and Gegenbauer polynomials.}},
\begin{equation}\label{sigmanjodd}
\sigma_{n,j}^{\alpha,\alpha}=0,\quad \text{for odd}\;\; j.
\end{equation}
\item[(ii)] For $\alpha=\beta=1/2$ {\rm(}Chebyshev polynomial of the second kind\,{\rm)},
\begin{equation}\label{constcase}
\sigma_{n,0}^{1/2,1/2}=\frac {\sqrt \pi} 2 \frac{(n+1)!}{\Gamma(n+3/2)}; \quad  \sigma_{n,j}^{1/2,1/2}=0,\quad \text{for}\;\; j\ge 1.
\end{equation}
\item[(iii)] For $\alpha=\beta=-1/2$ {\rm(}Chebyshev polynomial\,{\rm)},
\begin{equation}\label{constcase2}
\sigma_{n,j}^{-1/2,-1/2}=\begin{cases}
\dfrac{2\sqrt{\pi}\Gamma(n+1)}{\Gamma(n+1/2)},\quad & \text{for even}\;\; j, \\
0,\quad & \text{for odd}\;\; j.
\end{cases}
\end{equation}
\item[(iv)] For $\alpha=\beta=0$ {\rm(}Legendre polynomial\,{\rm)},
\begin{equation}\label{constcase3}
\sigma_{n,j}^{0,0}=\begin{cases}
\dfrac{2n+1}{2}\dfrac{\Gamma(l+1/2)}{\Gamma(l+1)}\dfrac{\Gamma(n+l+1)}{\Gamma(n+l+3/2)},\quad & \text{for even}\;\; j=2l, \\
0,\quad & \text{for odd}\;\; j.
\end{cases}
\end{equation}
\item[(v)]For general $\af, \bt>-1$ {\rm(}Jacobi polynomial\,{\rm)},
\begin{equation}\label{tempwa0}
\begin{split}
\sigma_{n,j}^{\af,\bt} &= \frac{\sqrt \pi
(2n+\af+\bt+1)\Gamma(n+\af+\bt+1)}
{2\Gamma(n+\af+1)}\\
&\quad \times \sum_{m=0}^{j}\frac{(-1)^m \Gamma(2n+j+m+2)
\Gamma(n+m+\af+1)}{m!(j-m)!\Gamma(n+m+3/2)\Gamma(2n+m+\af+\bt+2)}.
\end{split}
\end{equation}
\end{itemize}
\end{corollary}
\begin{proof} (i).  The property \eqref{sigmanjodd} is a direct consequence of the parity of  ultraspherical polynomials.

(ii). For $\alpha=\beta=1/2,$  we find from \eqref{chbytype2} and the orthogonality \eqref{jacobi_orth}-\eqref{gammafd} that
\begin{equation*}\label{sigmanj2}
\begin{split}
\sigma_{n,j}^{1/2,1/2}
&=\sqrt{\frac{\pi}{2}}\frac{1}{\sqrt{\gamma_{n+j}^{1/2,1/2}}}\frac{1}{\gamma_n^{1/2,1/2}}\int_{-1}^1J_{n+j}^{1/2,1/2}(x)J_n^{1/2,1/2}(x)(1-x^2)^{1/2}dx
\\&=\sqrt{\frac{\pi}{2}}\frac{1}{\sqrt{\gamma_{n+j}^{1/2,1/2}}}\delta_{j,0},
\end{split}
\end{equation*}
where $\delta_{j,0}$ is the Kronecker delta. Working out the constant leads to \eqref{constcase}.

(iii)  For $\af=\bt=-1/2$, if $j=2l,$ we have
\begin{equation*}\label{adf1f1}
\begin{split}
&\sigma_{n,2l}^{-1/2,-1/2}\overset{\re{chbytype2}}=
\frac{1}{\gamma_n^{-1/2,-1/2}}\frac 1
{n+2l+1}\int_{-1}^1T_{n+2l+1}'(x) J_n^{-1/2,-1/2}(x)(1-x^2)^{-1/2}
dx
\\&\qquad\overset{\re{dtn1}}=\frac{2}{\gamma_n^{-1/2,-1/2}}\int_{-1}^1T_n(x)J_n^{-1/2,-1/2}(x)(1-x^2)^{-1/2}
dx\overset{\re{dtn0}}=\frac{2\sqrt{\pi}\Gamma(n+1)}{\Gamma(n+1/2)},
 \end{split}
\end{equation*}
which, together with \eqref{sigmanjodd}, implies \eqref{constcase2}.

(iv) For $\af=\bt=0,$  we derive from \cite[Eq. (14)]{Chawla68} that
\begin{equation*}\label{sigmanj2a}
\begin{split}
\sigma_{n,2l}^{0,0}&=\frac{1}{\gamma_n^{0,0}}\int_{-1}^1J_{n}^{0,0}(x)U_{n+2l}(x)dx
=\frac{2n+1}{2}\int_0^{\pi}
J_{n}^{0,0}(\cos\theta)\sin\big((n+2l+1)\theta\big) d\theta
\\&=\frac{2n+1}{2}\frac{\Gamma(l+1/2)}{\Gamma(l+1)}\frac{\Gamma(n+l+1)}{\Gamma(n+l+3/2)},\quad
l\ge 0.
\end{split}
\end{equation*}
This yields \eqref{constcase3}.

(v) The formula \eqref{tempwa0} follows from a combination of \eqref{gammafd},  \eqref{cocoefng} and
\eqref{chbytype2}.
\end{proof}
%

With the aid of Lemma \ref{LM1}, we can derive the following  estimate, from which  our sharp bounds are stemmed.
\begin{lemma}\label{mainthm} For any $u\in {\mathcal A}_\rho$ with $\rho>1,$ we have that for
$\af,\bt>-1$ and $n\ge 0,$
\begin{equation}\label{expancoefest0}
\begin{split}
\big|\hat u_n^{\af,\bt}\big|\le \frac{M}{\rho^{n}}\Big(\big|\sigma_{n,0}^{\af,\bt}\big|
+\frac{1}{\rho}\big|\sigma_{n,1}^{\af,\bt}\big|
+\frac{1}{\rho^{2}} \sum_{j=0}^{\infty}\big|\sigma_{n,j+2}^{\af,\bt}-\sigma_{n,j}^{\af,\bt}\big|\frac{1}{\rho^{j}}\Big),
\end{split}
\end{equation}
where $M=\max_{z\in\mathcal{E}_\rho}|u(z)|$ and $\{\sigma_{n,j}^{\af,\bt}\}$ are given by \eqref{sigmanj}.
\end{lemma}
\begin{proof}
 Since $z=(w+w^{-1})/2\in {\mathcal E}_\rho$ with $w\in {\mathcal C}_\rho$ (cf. \eqref{crho}-\eqref{Berellips}),  we can rewrite  $\hat u_n^{\af,\bt}$ in  \eqref{expancoef} as
\begin{equation}\label{expancoepf1}
\begin{split}
\hat u_n^{\af,\bt}&
=\frac{1}{2\pi
\ri}\sum_{j=0}^{\infty}\sigma_{n,j}^{\af,\bt}\oint_{{\mathcal C}_\rho}\frac{u(z)}{w^{n+j+1}}\Big(1-\frac{1}{w^2}\Big)\,dw
\\&=\frac{1}{2\pi
\ri}\sum_{j=0}^{\infty}\sigma_{n,j}^{\af,\bt}\oint_{{\mathcal C}_\rho}\frac{u(z)}{w^{n+j+1}}\,dw
-\frac{1}{2\pi
\ri}\sum_{j=0}^{\infty}\sigma_{n,j}^{\af,\bt}\oint_{{\mathcal C}_\rho}\frac{u(z)}{w^{n+j+3}}\,dw
\\&=\frac{1}{2\pi
\ri}\sigma_{n,0}^{\af,\bt}\oint_{{\mathcal C}_\rho}\frac{u(z)}{w^{n+1}}\,dw+\frac{1}{2\pi
\ri}\sigma_{n,1}^{\af,\bt}\oint_{{\mathcal C}_\rho}\frac{u(z)}{w^{n+2}}\,dw
\\&\quad+\frac{1}{2\pi\ri}\sum_{j=0}^{\infty}\big(\sigma_{n,j+2}^{\af,\bt}-\sigma_{n,j}^{\af,\bt}\big)
\oint_{{\mathcal C}_\rho}\frac{u(z)}{w^{n+j+3}}\,dw.
\end{split}
\end{equation}
Hence,  we arrive at
\begin{equation}\label{expancoefest}
\begin{split}
\big|\hat u_n^{\af,\bt}\big|&\le
\frac{M}{2\pi}\frac{2\pi\rho}{\rho^{n+1}}\big|\sigma_{n,0}^{\af,\bt}\big|
+\frac{M}{2\pi}\frac{2\pi\rho}{\rho^{n+2}}\big|\sigma_{n,1}^{\af,\bt}\big|
+\frac{M}{2\pi}\frac{2\pi\rho}{\rho^{n+3}}
\sum_{j=0}^{\infty}\big|\sigma_{n,j+2}^{\af,\bt}-\sigma_{n,j}^{\af,\bt}\big|\frac{1}{\rho^{j}}
\\&=\frac{M}{\rho^{n}}\big|\sigma_{n,0}^{\af,\bt}\big|
+\frac{M}{\rho^{n+1}}\big|\sigma_{n,1}^{\af,\bt}\big|
+\frac{M}{\rho^{n+2}}
\sum_{j=0}^{\infty}\big|\sigma_{n,j+2}^{\af,\bt}-\sigma_{n,j}^{\af,\bt}\big|\frac{1}{\rho^{j}}.
\end{split}
\end{equation}
This ends the proof.
\end{proof}

Observe from the proof that we split the contour integral on  $\mathcal E_\rho$
into two parts on ${\mathcal C}_\rho,$ which actually  allows us to take the advantage of cancelation of $\sigma_{n,j+2}^{\af,\bt}-\sigma_{n,j}^{\af,\bt}.$
Indeed, the  bound \eqref{expancoefest0} is tight, as we will see shortly that this argument can recover the best estimate for
the Chebyshev case (see \cite[Theorem 3.8]{rivlin1990chebyshev} and \eqref{chebmars}), and improve the bounds  in \cite{xiang2012error} (see \eqref{xiangres}).

\subsection{Main results}\label{mainresults}
For clarity of exposition, we first present the result on the general Jacobi polynomial expansions, followed by the
refined results on the Chebyshev-type expansions ($\af=k-1/2,\bt=l-1/2$ with $k,l\in {\mathbb N}:=\{0,1,2,\cdots\}$), and
Legendre-type expansions ($\af=k,\bt=l$ with $k,l\in {\mathbb N}$).


\subsubsection{\bf General Jacobi expansions ($\af,\bt>-1$)}
\begin{thm}\label{jacobithm}
For  any $u\in {\mathcal A}_\rho$ {\rm(}with $\rho>1{\rm)}, \af,\bt>-1 $ and $n\ge 0,$  we
have
\begin{equation}\label{jacobiest1}
\begin{split}
\big|\hat
u_n^{\af,\bt}\big|&\le\frac{M}{\rho^{n}}\bigg[\big|\sigma_{n,0}^{\af,\bt}\big|+
\frac{|\sigma_{n,1}^{\af,\bt}|}{\rho}
+\frac{2}{\rho(\rho-1)}\sqrt{\frac{\gamma_0^{\af,\bt}}{\gamma_n^{\af,\bt}}}\; \bigg],
\end{split}
\end{equation}
 where
\begin{equation}\label{sigman0}
\sigma_{n,0}^{\af,\bt}=\frac{\sqrt \pi} 2 \frac{(2n+1)!\Gamma(n+\af+\bt+1)}{\Gamma(n+3/2)\Gamma(2n+\af+\bt+1)},
\quad \sigma_{n,1}^{\af,\bt}=\frac{(\bt-\af)(2n+2)}{2n+\af+\bt+2}\sigma_{n,0}^{\af,\bt},
\end{equation}
and $\gamma_n^{\af,\bt}$ is defined in \eqref{gammafd}.

In particular, if $\af=\bt$, we have
\begin{equation}\label{gegenest1}
\begin{split}
\big|\hat
u_n^{\af,\af}\big|&\le\frac{M}{\rho^{n}}\bigg[\big|\sigma_{n,0}^{\af,\af}\big|
+\frac{2}{\rho^2-1}\sqrt{\frac{\gamma_0^{\af,\af}}
{\gamma_n^{\af,\af}}}\; \bigg].
\end{split}
\end{equation}
\end{thm}
\begin{proof}
By \eqref{expancoefest0},
\begin{equation}\label{jacobiesta}
\begin{split}
\big|\hat u_n^{\af,\bt}\big|&\le
\frac{M}{\rho^{n}}\big|\sigma_{n,0}^{\af,\bt}\big|
+\frac{M}{\rho^{n+1}}\big|\sigma_{n,1}^{\af,\bt}\big|
+\frac{M}{\rho^{n+2}}
\sum_{j=0}^{\infty}\big|\sigma_{n,j+2}^{\af,\bt}-\sigma_{n,j}^{\af,\bt}\big|\frac{1}{\rho^{j}}.
\end{split}
\end{equation}
The factors $\sigma_{n,0}^{\af,\bt}$ and $\sigma_{n,1}^{\af,\bt}$ in \eqref{sigman0}  are
computed from \eqref{tempwa0} directly, so it suffices to estimate the infinite sum in \eqref{jacobiesta}.
Recall the identity (cf. \cite{Mason03}):
\begin{equation}\label{UTrela}
U_k(x)-U_{k-2}(x)=2T_k(x),\quad k\ge 2.
\end{equation}
Then we infer from  \eqref{sigmanj} that
\begin{equation}\label{sigmanj3ab}
\begin{split}
\sigma_{n,j+2}^{\af,\bt}-\sigma_{n,j}^{\af,\bt}&= \frac{1}{\gamma_n^{\af,\bt}}\int_{-1}^1\big(U_{n+j+2}(x)-U_{n+j}(x)\big)J_n^{\af,\bt}(x)\omega^{\af,\bt}(x)\,dx
\\&=\frac{2}{\gamma_n^{\af,\bt}}
\int_{-1}^1T_{n+j+2}(x)J_n^{\af,\bt}(x)\omega^{\af,\bt}(x)\,dx,\quad n, j\ge 0.
\end{split}
\end{equation}
Thus, using the Cauchy-Schwartz inequality, the orthogonality \eqref{jacobi_orth}, and the fact $|T_k(x)|\le 1,$ leads to
\begin{equation}\label{sigmanj3a}
\big|\sigma_{n,j+2}^{\af,\bt}-\sigma_{n,j}^{\af,\bt}\big| \le
\frac 2{\sqrt{\gamma_n^{\af,\bt}}}\Big(\int_{-1}^1 T_{n+j+2}^2(x) \omega^{\af,\bt}(x)\, dx \Big)^{1/2}\le 2\sqrt{\frac{\gamma_0^{\af,\bt}}{\gamma_n^{\af,\bt}}}.
\end{equation}
Therefore, the bound \eqref{jacobiest1} follows from $\sum_{j=0}^{\infty}\rho^{-j}=1/(1-\rho^{-1}),$ as $\rho>1.$

 For $\af=\bt,$ since $|\sigma_{n,2l+1}^{\af,\af}|=0,$ for all $l\ge0$ (cf. Corollary \ref{sigmanjcor}
 (i)), we have
 \[
\sum_{j=0}^{\infty}\big|\sigma_{n,j+2}^{\af,\af}-\sigma_{n,j}^{\af,\af}\big|\frac{1}{\rho^{j}}=
\sum_{l=0}^{\infty}\big|\sigma_{n,2l+2}^{\af,\af}-\sigma_{n,2l}^{\af,\af}\big|\frac{1}{\rho^{2l}}
\le2\sqrt{\frac{\gamma_0^{\af,\af}}{\gamma_n^{\af,\af}}}\frac{1}{1-\rho^{-2}}.
 \]
This yields the refined bound in \eqref{gegenest1}.
\end{proof}

\begin{rem}\label{astotpic} Using  Lemma \ref{lemGammaratio}, we can characterize the explicit  dependence of the upper bounds in
\eqref{jacobiest1} and \eqref{gegenest1} on $n,\alpha,\beta.$ Indeed,
   for  $\af,\bt>-1, n\ge 1$ and $n+\alpha+\beta>0, $
\begin{equation}\label{asymptcons1}
\begin{split}
\sigma_{n,0}^{\af,\bt}&\overset{(\ref{sigman0})}=\frac{\sqrt \pi} 2
\frac{\Gamma(n+\af+\bt+1)}{\Gamma(n+3/2)} \frac{(2n+1)!}{\Gamma(2n+\af+\bt+1)}
\\&\overset{(\ref{Gammaratio})}\le \frac{\sqrt \pi} 2\big(\Upsilon_{n}^{\af+\bt+1,3/2}n^{\af+\bt+1-3/2}\big) \big(\Upsilon_{2n}^{2,\af+\bt+1}(2n)^{2-(\af+\bt+1)}\big)
\\&=\frac{\sqrt{\pi n}}{{2^{\af+\bt}}}\Upsilon_{n}^{\af+\bt+1,3/2}\Upsilon_{2n}^{2,\af+\bt+1}\overset{(\ref{consUps})}=\frac{\sqrt{\pi n}}{{2^{\af+\bt}}}\big(1+O(n^{-1})\big),
\end{split}
\end{equation}
which implies
\begin{equation}\label{asymptcons2}
\begin{split}
|\sigma_{n,1}^{\af,\bt}|&\overset{(\ref{sigman0})}=\frac{|\alpha-\beta|(2n+2)}{2n+\af+\bt+2}\sigma_{n,0}^{\af,\bt}\le\frac{|\alpha-\beta|(2n+2)}{2n+\af+\bt+2}\frac{\sqrt{\pi
n}}{{2^{\af+\bt}}}\Upsilon_{n}^{\af+\bt+1,3/2}\Upsilon_{2n}^{2,\af+\bt+1}\\
&=|\af-\bt|\frac{\sqrt{\pi n}}{{2^{\af+\bt}}}\big(1+O(n^{-1})\big).
\end{split}
\end{equation}
Similarly, one verifies 
\begin{equation}\label{asymptcons3}
\begin{split}
\frac{\gamma_0^{\af,\bt}}{\gamma_n^{\af,\bt}} &\overset{(\ref{gammafd})}=(2n+\alpha+\beta+1)
\frac{\gamma_0^{\af,\bt}}{2^{\af+\bt+1}}\frac{n!\Gamma(n+\alpha+\beta+1)}{\Gamma(n+\alpha+1)\Gamma(n+\beta+1)}
\\&\le(2n+\alpha+\beta+1)
\frac{\Gamma(\alpha+1)\Gamma(\beta+1)}{\Gamma(\alpha+\beta+2)}\Upsilon_n^{1,\alpha+1}\Upsilon_n^{\alpha+\beta+1,\beta+1}\\
&\overset{(\ref{consUps})}=\frac{2\Gamma(\alpha+1)\Gamma(\beta+1)}{\Gamma(\alpha+\beta+2)} n \big(1+O(n^{-1})\big).
\end{split}
\end{equation}
Consequently,  we infer from the estimate \eqref{jacobiest1} that for fixed $\af,\bt>-1$ and $n\gg 1,$  
\begin{equation}\label{asympJacobi}
|\hat u_n^{\af,\bt}|\le  C_n M\bigg(\frac{\sqrt{\pi
}}{2^{\af+\bt}}\Big(1+\frac{|\af-\bt|}{\rho}\Big)+
\sqrt{\frac{\Gamma(\alpha+1)\Gamma(\beta+1)}{\Gamma(\alpha+\beta+2)}}\frac{2\sqrt{2}}{\rho(\rho-1)}
\bigg)\frac{\sqrt n}{\rho^n},
\end{equation}
and likewise,  we find from \eqref{gegenest1} that
\begin{equation}\label{asympJacobi2}
|\hat u_n^{\af,\af}|\le C_n M \bigg(\frac{\sqrt{\pi
}}{2^{2\af}}+
\frac{\Gamma(\alpha+1)}{\sqrt{\Gamma(2\alpha+2)}}\frac{2\sqrt{2}}{\rho^2-1}
\bigg)\frac{\sqrt n}{\rho^n},
\end{equation}
where $C_n=1+O(n^{-1}).$ \qed
\end{rem}

\begin{rem}\label{tightness} It is worthwhile to show  that the bound obtained in this way is tighter than \eqref{xiangres} obtained in \cite{xiang2012error}. Indeed, it follows from \eqref{chbytype2}, \eqref{dtn1} and \eqref{jacobi_orth} that for $n\ge 1$ and $j=0,1,$
\begin{align*}
\sigma_{n,j}^{\af,\bt}&=\frac{1}{\gamma_n^{\af,\bt}}\frac 1
{n+j+1}\int_{-1}^1 T_{n+j+1}'(x)
J_n^{\af,\bt}(x)\omega^{\af,\bt}(x)dx\\& =
\frac{2}{\gamma_n^{\af,\bt}}
\underset{k+n+j+1\;\text{odd}}{\underset{k=0}{\sum^{n+j}}}
\frac1{c_k}\int_{-1}^1 T_k(x) J_n^{\af,\bt}(x)\omega^{\af,\bt}(x)dx\\
&=  \frac{2}{\gamma_n^{\af,\bt}}  \int_{-1}^1 T_{n+j}(x) J_n^{\af,\bt}(x)\omega^{\af,\bt}(x)dx,
 \end{align*}
 where $c_0=2$ and $c_k=1$ for $k\ge 1.$
 Following \eqref{sigmanj3ab}-\eqref{sigmanj3a}, we have
 $$|\sigma_{n,j}^{\af,\bt}|\le 2 \sqrt{\frac{\gamma_0^{\af,\bt}} {\gamma_n^{\af,\bt}}},\quad n\ge 1,\;\; j=0,1. $$
Finally,  a straightforward calculation leads to
\begin{equation}\label{betterest}
\frac{M}{\rho^{n}}\bigg[\big|\sigma_{n,0}^{\af,\bt}\big|+
\frac{|\sigma_{n,1}^{\af,\bt}|}{\rho}
+\frac{2}{\rho(\rho-1)}\sqrt{\frac{\gamma_0^{\af,\bt}}{\gamma_n^{\af,\bt}}}\; \bigg]\le \frac{2M}{\rho^{n-1}(\rho -1)}\sqrt{\frac{\gamma_0^{\af,\bt}}{\gamma_n^{\af,\bt}}}.
\end{equation}
Moreover, we claim from \eqref{gegenest1} that the strict inequality holds, when $\af=\bt>-1.$  One may refer to
Section \ref{sect4} for  numerical evidences.  \qed
\end{rem}

\vskip 6pt

\subsubsection{\bf Chebyshev-type expansions ($\af=k-1/2,\bt=l-1/2$ with $k,l\in {\mathbb N}$)}~~~

\vskip4pt

In view of  \eqref{constcase2}, it follows from  \eqref{expancoepf1} that the Chebyshev coefficient takes the simplest form:
\begin{equation}\label{chbydecay0}
\hat u_n^{-1/2,-1/2}= \frac{\sigma_{n,0}^{-1/2,-1/2}}{2\pi \ri}\oint_{{\mathcal C}_\rho}\frac{u(z)}{w^{n+1}}\,dw.
\end{equation}
Thus, using  \eqref{constcase2} and \eqref{expancoefest0} leads to
\begin{equation}\label{chbydecay}
|\hat u_n^{-1/2,-1/2}|\le \dfrac{2\sqrt{\pi}\Gamma(n+1)}{\Gamma(n+1/2)}\frac{M}{\rho^n}.
\end{equation}
This leads to the estimate for the expansion coefficients, denoted by $\{\hat u_n^C\}$ as before,  in terms of
 $\{T_n(x)\}:$
 \begin{equation}\label{chebyest3}
\big|\hat u_n^C\big|\le \frac{2M}{\rho^{n}},\quad n\ge 0,
\end{equation}
as documented in e.g., \cite{rivlin1990chebyshev}.

For the second-kind Chebyshev case,  we find from \eqref{expancoef}
  the closed-form formula like \eqref{chbydecay0}:
\begin{align}\label{2ndcheb1}
\hat u_n^{1/2,1/2}=\frac{\sigma_{n,0}^{1/2,1/2}}{\pi
\ri}\oint_{\mathcal{E}_\rho}\frac{u(z)}{w^{n+1}}\,dz,
\end{align}
but the contour integration is on ${\mathcal E}_\rho.$ It follows
from  \eqref{constcase} and  \eqref{expancoepf1} that
\begin{align}\label{estcheb2nd}
|\hat u_n^{1/2,1/2}|\le \frac 1 {2 \sqrt \pi }
\frac{(n+1)!}{\Gamma(n+3/2)}
\Big|\oint_{\mathcal{E}_\rho}\frac{u(z)}{w^{n+1}}\,dz\Big|\le
\frac{\sqrt \pi} 2 \frac{(n+1)!}{\Gamma(n+3/2)}\frac M
{\rho^n}{\Big(1+\frac{1}{\rho^2}\Big)}.
\end{align}
Like \eqref{chebyest3}, if we re-scale the expansion in terms of $\{U_n\},$ i.e.,
$$\hat u_n^U=\frac 2 \pi \int_{-1}^1 u(x) U_n(x) \sqrt{1-x^2}\, dx,  $$
then we find from \eqref{chbytype2} and \eqref{estcheb2nd} that
\begin{align}\label{estcheb2nd1}
\big|\hat
u_n^U\big|=\frac{2}{\sqrt{\pi}}\frac{\Gamma(n+3/2)}{\Gamma(n+2)}|\hat
u_n^{1/2,1/2}|\le  \frac M
{\rho^n}{\Big(1+\frac{1}{\rho^2}\Big)}.
\end{align}
\begin{rem}\label{2ndchbrem} It is seen from  \eqref{2ndcheb1} that the second-kind Chebyshev coefficient takes the simplest form
on the contour $\mathcal E_\rho.$ This motivates us to estimate the contour integral directly by
$$\Big|\oint_{\mathcal{E}_\rho}\frac{u(z)}{w^{n+1}}\,dz\Big|\le \frac{M}{\rho^{n+1}}\oint_{\mathcal E_\rho} |dz|=\frac{M}{\rho^{n+1}}L({\mathcal E}_\rho),  $$
which implies
\begin{align}\label{estcheb2nd2}
\big|\hat
u_n^U\big|\le  \frac{M}{\rho^{n+1}} \frac{L({\mathcal E}_\rho)}{\pi}.
\end{align}
By \eqref{Lehp},
$$ \frac {L({\mathcal E}_\rho)}{\pi \rho}\le \sqrt{1+\frac 1 {\rho^4}}<1+\frac{1}{\rho^2}.$$
Therefore, the estimate  \eqref{estcheb2nd2} is slightly sharper than \eqref{estcheb2nd1}. \qed
\end{rem}

Some refined results can also be derived for  $\af=k+1/2, \bt=l+1/2$ with $k,l\in\mathbb{N}.$   Indeed, we find that
$\big\{\sigma_{n,j}^{k+1/2,l+1/2}\big\}$ can be computed  explicitly by the following formula.
 \begin{prop}\label{chebtypelem} For any $k,l, n,j \in\mathbb{N},$
 \begin{equation}\label{chebyest21}
\begin{split}
\sigma_{n,j}^{k+1/2,l+1/2}=\sqrt{\frac{\pi}{2}} \frac{1}{\gamma_n^{k+1/2,l+1/2}}\sum_{m=n}^{n+k+l}d_{m}^{k+1/2,l+1/2}\sqrt{\gamma_{m}^{1/2,1/2}}\delta_{m,n+j},
\end{split}
\end{equation}
where $\big\{d_{m}^{k+1/2,l+1/2}\big\}_{m=n}^{n+k+l}$ are given in \eqref{compexps}, and
$\delta_{m,n+j}$ is the Kronecker delta.
 \end{prop}
 \begin{proof} Using  \eqref{compexps} (with $\alpha=\beta=1/2$), \eqref{sigmanj} and the properties of  Jacobi polynomials (cf. \eqref{jacobi_orth} and \eqref{chbytype2}),  leads to
\begin{equation}\label{chebyest2}
\begin{split}
\sigma_{n,j}^{k+1/2,l+1/2}&=\frac{1}{\gamma_n^{k+1/2,l+1/2}}\sum_{m=n}^{n+k+l}d_{m}^{k+1/2,l+1/2}\int_{-1}^1U_{n+j}(x)J_m^{1/2,1/2}(x)(1-x^2)^{1/2}dx
\\&=\frac{1}{\gamma_n^{k+1/2,l+1/2}}\sqrt{\frac{\pi}{2}}\sum_{m=n}^{n+k+l}d_{m}^{k+1/2,l+1/2}\sqrt{\gamma_{m}^{1/2,1/2}}\delta_{m,n+j},
\end{split}
\end{equation}
This completes the proof.
\end{proof}

Equipped with \eqref{chebyest21}, we can obtain the bound for Chebyshev-type expansion coefficients
 by  computing $\{d_{m}^{k+1/2,l+1/2}\}$ explicitly. To fix the idea, we just consider the case: $k=1$ and $l=0.$
 One finds
$$d_{n}^{3/2,1/2}=1,\quad d_{n+1}^{3/2,1/2}=-\frac{2n+2}{ 2n+3},$$
and
\begin{equation*}
\sigma_{n,0}^{3/2,1/2}=\frac{\sqrt{\pi}}{4}\frac{n!(n+2)}{\Gamma(n+3/2)},\quad
\sigma_{n,1}^{3/2,1/2}=-\frac{\sqrt{\pi}}{4}\frac{(n+1)!}{\Gamma(n+3/2)},\quad  \sigma_{n,j}^{3/2,1/2}=0,\;\; j\ge 2.
\end{equation*}
The estimate \eqref{jacobicoe} reduces to
\begin{equation*}
\begin{split}
\hat u_n^{3/2,1/2}&\le
\frac{M}{\rho^n}\bigg[\sigma_{n,0}^{3/2,1/2}+\frac{\sigma_{n,1}^{3/2,1/2}}{\rho}+
\frac{\sigma_{n,0}^{3/2,1/2}}{\rho^2}+\frac{\sigma_{n,1}^{3/2,1/2}}{\rho^3}\bigg]
\\&=\frac{M}{\rho^n}\Big(1+\frac{1}{\rho^2}\Big)\bigg[\sigma_{n,0}^{3/2,1/2}+\frac{\sigma_{n,1}^{3/2,1/2}}{\rho}\bigg].
\end{split}
\end{equation*}
Thus, we have
\begin{equation}\label{estchebexample0}
|\hat{u}_n^{3/2,1/2}|\le\frac{\sqrt{\pi}}{4}\frac{(n+1)!}{\Gamma(n+3/2)}
\frac{M}{\rho^n}\Big(1+\frac{1}{\rho^2}\Big)\Big(\frac{n+2}{n+1}+\frac{1}{\rho}\Big),
\end{equation}
 and by \eqref{Gammaratio}, we have for $n\ge0,$
\begin{equation}\label{chebytyperatio}
\frac{(n+1)!}{\Gamma(n+3/2)}\le\sqrt{n}\exp\Big(\frac{8n+7}{12(2n+1)(n+1)}+\frac{1}{4n}\Big).
\end{equation}

 Actually,  the infinite sum  in \eqref{expancoefest0} does not appear for the
Chebyshev-type expansions, which  allows us to derive very tight
bounds. However, for the Legendre-type expansions, some care has to
be taken to handle this sum.

\vspace*{6pt}

\subsubsection{\bf Legendre-type expansions ($\af=k,\bt=l$ with $k,l\in {\mathbb N}$)}~~~

\vskip4pt

We first consider the Legendre case.   By \eqref{gammafd}
and \eqref{sigman0},
\[
\gamma_n^{0,0}=\frac{2}{2n+1},\quad
\frac{\gamma_0^{0,0}}{\gamma_{n}^{0,0}}=2n+1, \quad
\sigma_{n,0}^{0,0}=\frac{\sqrt{\pi}\Gamma(n+1)}{\Gamma(n+1/2)},
\]
so the estimate  \eqref{gegenest1} reduces to
\begin{equation}\label{legen}
\big|\hat u_n^{0,0}\big|\le
\frac{M}{\rho^n}\Big[\frac{\sqrt{\pi}\Gamma(n+1)}{\Gamma(n+1/2)}+\frac{2\sqrt{2n+1}}{\rho^2-1}\Big].
\end{equation}

In fact, we can improve this  estimate, as highlighted in the following theorem,  by using the explicit information of  $\sigma_{n,2l}^{0,0}.$
\begin{thm}\label{Th:legexp} Let $\{\hat u_n^{0,0}\}$ be the Legendre expansion coefficients of any $u\in {\mathcal A}_\rho$ with $\rho>1.$ Then for any $n\ge 1,$
\begin{equation}\label{legenest4}
\big|\hat{u}_n^{0,0}\big|\le\frac{M\sqrt{\pi
n}}{\rho^{n}}\Big(1+\frac{n+2}{2n+3}\frac{1}{\rho^2-1}\Big)\exp\Big(\frac{8n-1}{12n(2n-1)}\Big).
\end{equation}
\end{thm}
\begin{proof} A straightforward calculation from
\eqref{constcase3} yields
\begin{equation}\label{sigmanj2ca0}
\sigma_{n,2l+2}^{0,0}-\sigma_{n,2l}^{0,0}=-\frac{n+2l+2}{2(l+1)(n+l+3/2)}\sigma_{n,2l}^{0,0},\quad
 l\ge 0,
\end{equation}
which implies  $\{\sigma_{n,2l}^{0,0}\}$ is strictly descending with
respect to $l.$ Hence, we have
\begin{equation}\label{sigmanj2ca}
\big|\sigma_{n,2l+2}^{0,0}-\sigma_{n,2l}^{0,0}\big|
=\frac{n+2l+2}{2(l+1)(n+l+3/2)}\sigma_{n,2l}^{0,0} \le
\frac{n+2}{2n+3}\sigma_{n,0}^{0,0},
\end{equation}
where we used the fact that ${n+2l+2}/{((l+1)(n+l+3/2))}$ is
strictly descending with respect to $l.$ Then,  we obtain the
improved bound from \eqref{expancoefest0}:
\begin{equation}\label{legenest3}
\begin{split}
\big|\hat{u}_n^{0,0}\big|&\le\frac{M}{\rho^{n}}\sigma_{n,0}^{0,0}\Big(1+
\frac{n+2}{2n+3}\sum_{l=0}^{\infty}\frac{1}{\rho^{2l+2}}\Big)
=\frac{\sqrt{\pi}
M}{\rho^{n}}\frac{\Gamma(n+1)}{\Gamma(n+1/2)}\Big(1+\frac{n+2}{2n+3}\frac{1}{\rho^2-1}\Big),
\end{split}
\end{equation}
and by \eqref{Gammaratio},
\begin{equation}\label{legensigmaest}
\frac{\Gamma(n+1)}{\Gamma(n+1/2)}\le
\sqrt{n}\exp\Big(\frac{8n-1}{12n(2n-1)}\Big),\quad n\ge 1.
\end{equation}
This completes the proof.
\end{proof}


\begin{rem}\label{legenrem} We compare the bound in \eqref{legenest4} with the existing ones.
Davis \cite[Page 313]{davis1975interpolation} stated the bound
$$\big|\hat{u}_n^{0,0}\big|\le \frac{2n+1} 2 \frac{M L({\mathcal E}_\rho)} {\rho^n (\rho-1)} \overset{(\ref{Lehp})}\le
\frac{2n+1} 2 \frac{\pi \sqrt{\rho^2+\rho^{-2}} M } {\rho^n (\rho-1)}, $$
where clearly the algebraic order of $n$ in the numerator is not optimal. The following asymptotic
bound can be obtained from \cite[Eq. (32) and Eq. (38)]{Kambo70} and \cite[Eq. (12.4.25)]{davis1975interpolation}:
\begin{equation*}\label{davisbnd}
\big|\hat{u}_n^{0,0}\big|\le  \frac{M\sqrt{\pi
n}}{\rho^{n}} \frac{\sqrt{\rho^4+1}}{\rho^2-1},\quad n\gg 1,
\end{equation*}
while the asymptotic estimate derived  from \eqref{legenest4} is
\begin{equation}\label{Kamboestadded}
\big|\hat{u}_n^{0,0}\big|\le  \frac{M\sqrt{\pi n}}{\rho^{n}} \frac{{\rho^2-1/2}}{\rho^2-1},\quad n\gg 1,
\end{equation}
which is sharper. Another bound for comparison is obtained in the recent paper \cite{xiang2012error}:
 \begin{equation}\label{xianglegenbnd}
\big|\hat{u}_n^{0,0}\big|\le
\frac{2\sqrt{n}M}{\rho^{n}}\Big(1+\frac{1}{\rho^2-1}\Big),\quad n\ge
1,
\end{equation}
which is also inferior to  our estimate  \eqref{legenest4}.  Some comparisons in numerical perspective  are given in
Section \ref{sect4}.
 \qed
\end{rem}

 Like the Chebysheve case, we can derive similar refined estimates for  Legendre-type expansions with $\alpha=k, \beta=l$ and $k,l\in {\mathbb N}.$  The counterpart of Proposition \ref{chebtypelem} is stated as follows, which can be obtained by using   \eqref{compexps} (with $\alpha=\beta=0$), \eqref{sigmanj} and the properties of  Jacobi polynomials (e.g., \eqref{jacobi_orth}) as before.
 \begin{prop}\label{legentypelem} For any $k,l, n,j \in\mathbb{N},$
 \begin{equation}\label{legenest21}
\sigma_{n,j}^{k,l}=\frac{1}{\gamma_n^{k,l}}\sum\limits_{m=n}^{n+k+l}d_{m}^{k,l}\gamma_m^{0,0}\sigma_{m,n+j-m}^{0,0},
\end{equation}
where $\big\{d_{m}^{k,l}\big\}_{m=n}^{n+k+l}$ are the same as in
\eqref{compexps}, and $\big\{\sigma_{m,n+j-m}^{0,0}\big\}$ are computed by
\eqref{constcase3}.
 \end{prop}

Once again,  to fix the idea, we just consider the case: $k=1$ and $l=0.$  One finds
$d_{n}^{1,0}=1, d_{n+1}^{1,0}=-1,$ and
\begin{equation*}
\begin{split}
\sigma_{n,j}^{1,0}&=
\frac{1}{\gamma_n^{1,0}}\big(\gamma_n^{0,0}\sigma_{n,j}^{0,0}-\gamma_{n+1}^{0,0}\sigma_{n+1,j-1}^{0,0}\big)
=\frac{n+1}{2n+1}\sigma_{n,j}^{0,0}-\frac{n+1}{2n+3}\sigma_{n+1,j-1}^{0,0}.
\end{split}
\end{equation*}
By \eqref{constcase3},
\[
\sigma_{n,2l}^{1,0}=\frac{n+1}{2n+1}\sigma_{n,2l}^{0,0},\quad
\sigma_{n,2l+1}^{1,0}=-\frac{n+1}{2n+3}\sigma_{n+1,2l}^{0,0},\quad
l\ge0.
\]
Therefore, with \eqref{sigmanj2ca0} and  \eqref{sigmanj2ca}, the
estimate \eqref{expancoefest0} reduces to
\begin{equation*}\label{legenest1}
\begin{split}
\big|\hat u_n^{1,0}\big| &\le
\frac{M}{\rho^{n}}\big|\sigma_{n,0}^{1,0}\big|
+\frac{M}{\rho^{n+1}}\big|\sigma_{n,1}^{1,0}\big|
+\frac{M}{\rho^{n+2}}
\sum_{j=0}^{\infty}\big|\sigma_{n,j+2}^{1,0}-\sigma_{n,j}^{1,0}\big|\frac{1}{\rho^{j}}
\\&=\frac{M}{\rho^{n}}\frac{n+1}{2n+1}\Big(\sigma_{n,0}^{0,0}
+\frac{1}{\rho^{2}}
\sum_{l=0}^{\infty}\big|\sigma_{n,2l+2}^{0,0}-\sigma_{n,2l}^{0,0}\big|\frac{1}{\rho^{2l}}\Big)
\\&\quad+\frac{M}{\rho^{n+1}}\frac{n+1}{2n+3}\Big(\sigma_{n+1,0}^{0,0}+\frac{1}{\rho^{2}}
\sum_{l=0}^{\infty}\big|\sigma_{n+1,2l+2}^{0,0}-\sigma_{n+1,2l}^{0,0}\big|\frac{1}{\rho^{2l}}\Big)
\\&\le\sigma_{n,0}^{0,0}\frac{n+1}{2n+1}\frac{M}{\rho^{n}}\Big(1+\frac{n+2}{2n+3}\frac{1}{\rho^2-1}\Big)
+\sigma_{n+1,0}^{0,0}\frac{n+1}{2n+3}\frac{M}{\rho^{n+1}}\Big(1+\frac{n+3}{2n+5}\frac{1}{\rho^2-1}\Big).
\end{split}
\end{equation*}
Working out the expressions of $\sigma_{n,0}^{0,0}$ and
$\sigma_{n+1,0}^{0,0}$ by \eqref{sigman0}, we have
\begin{equation}\label{legenest10}
\begin{split}
\big|\hat u_n^{1,0}\big|&
\le\frac{M}{\rho^{n}}\frac{\sqrt{\pi}\Gamma(n+2)}{\Gamma(n+3/2)}\Big\{\frac{1}{2}+\frac{n+2}{2(2n+3)}\frac{1}{\rho^2-1}
+\frac{1}{\rho}\frac{n+1}{2n+3}\Big(1+\frac{n+3}{2n+5}\frac{1}{\rho^2-1}\Big)\Big\}.
\end{split}
\end{equation}
Note that the ratio of the Gamma functions can be bounded as in
\eqref{chebytyperatio}.

The same process applies to  other  $k, l\in {\mathbb N},$  but the
derivation seems tedious.

\vskip 6pt

\subsection{Estimates for truncated Jacobi expansions} Given  a cut-off number $N\ge 1$ and  $N\in {\mathbb N},$  we define
the partial sum
\begin{equation}\label{partialsum}
\big(\pi_{N}^{\af,\bt} u\big)(x)=\sum_{n=0}^{N-1} \hat u_n^{\af,\bt}
J_n^{\af,\bt}(x),
\end{equation}
where $\{\hat u_n^{\af,\bt}\}$ are the Jacobi expansion coefficients defined in \eqref{jacobicoe}.
To this end, let $L^2_{\omega^{\af,\bt}}(I)$ be the weighted $L^2$-space on $I=(-1,1),$  and its norm is denoted by
$\|\cdot\|_{\omega^{\af,\bt}},$ where we drop the weight function, if $\af=\bt=0.$

Notice that  $\pi_{N}^{\af,\bt} u$ is the $L^2_{\omega^{\af,\bt}}$-projection of $u$ upon $P_{N-1}$ (denoting the set of all
algebraic polynomials of degree at most $N-1$), that is,
$\pi_{N}^{\af,\bt} u$ is the best approximation to $u$ in the norm $\|\cdot\|_{\omega^{\af,\bt}}.$
With the previous bounds for the expansion coefficients, we can estimate the truncation error straightforwardly.
\begin{thm}\label{coeffdecay} For any
 $u\in {\mathcal A}_\rho$ with $\rho>1,$ and $\af,\bt>-1,$ we have
\begin{equation}\label{L2errorest}
\begin{split}
\big\|\pi_N^{\alpha,\beta}u-u\big\|_{\omega^{\alpha,\beta}}&\le
\bigg[\sqrt{\frac{\pi }{2^{\af+\bt}}}
\Big(1+\frac{|\af-\bt|}{\rho}\Big)+\frac{2\sqrt{\gamma_0^{\af,\bt}}}{\rho(\rho-1)}\;\bigg]\frac{C_N M}{\rho^{N-1}\sqrt{\rho^2-1}},
\end{split}
\end{equation}
where $\gamma_0^{\af,\bt}$ is given in \eqref{gammafd} and
$C_N \approx 1.$
\end{thm}
\begin{proof}
By the orthogonality (cf \eqref{jacobi_orth}-\eqref{gammafd}) of
Jacobi polynomials, we have
\[
  \begin{split}
\big\|\pi_N^{\alpha,\beta}u-u\big\|_{\omega^{\alpha,\beta}}^2&=\sum_{n=N}^\infty
|\hat u_n^{\alpha,\beta}|^2\gamma_n^{\alpha,\beta}.
\end{split}
\]
It follows from the estimate of $\big|\hat u_n^{\af,\bt}\big|$ in Theorem
\ref{jacobithm}, and a combination of \eqref{gammacoef}-\eqref{consUps} and \eqref{asymptcons1}-\eqref{asymptcons2} that for  $n\ge N\gg 1$,
\begin{equation*}
\begin{split}
\big|\hat
u_n^{\af,\bt}\big|\sqrt{\gamma_n^{\alpha,\beta}}&\le\frac{M}{\rho^{n}}\bigg[\big|\sigma_{n,0}^{\af,\bt}\big|\sqrt{\gamma_n^{\alpha,\beta}}+
\frac{1}{\rho}\big|\sigma_{n,1}^{\af,\bt}\big|\sqrt{\gamma_n^{\alpha,\beta}}
+\frac{2}{\rho(\rho-1)}\sqrt{\gamma_0^{\af,\bt}}\;\bigg]
\\&\le\frac{C_n M}{\rho^{n}}\bigg[\sqrt{\frac{\pi}{2^{\af+\bt}}}
\Big(1+\frac{|\af-\bt|}{\rho}\Big)+\frac{2}{\rho(\rho-1)}\sqrt{\gamma_0^{\af,\bt}}\;\bigg],
\end{split}
\end{equation*}
where $C_n=1+O(n^{-1}).$
Therefore, we have
\begin{equation*}
  \begin{split}
\big\|\pi_N^{\alpha,\beta}u-u\big\|_{\omega^{\alpha,\beta}}&\le
C_N M\bigg[\sqrt{\frac{\pi }{2^{\af+\bt}}}
\Big(1+\frac{|\af-\bt|}{\rho}\Big)+\frac{2\sqrt{\gamma_0^{\af,\bt}}}{\rho(\rho-1)}\;\bigg]
\Big(\sum_{n=N}^\infty\frac{1}{\rho^{2n}}\Big)^{1/2}
\\&\le \bigg[\sqrt{\frac{\pi }{2^{\af+\bt}}}
\Big(1+\frac{|\af-\bt|}{\rho}\Big)+\frac{2\sqrt{\gamma_0^{\af,\bt}}}{\rho(\rho-1)}\;\bigg]\frac { C_NM} {\rho^{N-1}\sqrt{\rho^2-1}}.
\end{split}
\end{equation*}
This ends the proof.
\end{proof}

\begin{rem}\label{highorderest} Note that $\{\frac {d^l}{dx^l} J_n^{\af,\bt}\}_{n\ge l}$ are mutually orthogonal with respect to
$\omega^{\af+l,\bt+l},$ so we can estimate $\big\|(\pi_N^{\alpha,\beta}u-u)^{(l)}\big\|_{\omega^{\alpha+l,\beta+l}}$
in a similar fashion. \qed
\end{rem}

\begin{rem}\label{legremark} Some refined estimates can be obtained from the refined bounds for special cases, e.g.,  $\af=\bt$ or $\af=\bt=0,-1/2.$  Here, we just state the result for the Legendre case:
\begin{equation}\label{L2errorestLeg}
\begin{split}
\big\|\pi_N^{0,0}u-u\big\|\le \Big(1+\frac{1}{2(\rho^2-1)}\Big)\frac{C_N \sqrt \pi M}{\rho^{N-1}\sqrt{\rho^2-1}},
\end{split}
\end{equation}
where $C_N\approx 1$ as before. It follows from Theorem \ref{Th:legexp} and the above process. 
Note that Xiang  \cite{xiang2012error} derived the following estimate for the Legendre expansion:
\begin{equation}\label{xiangL2errorestLeg}
\begin{split}
\big\|\pi_N^{0,0}u-u\big\|\le\frac{2\sqrt{2}M}{\rho^{N-2}(\rho-1)^2}.
\end{split}
\end{equation}
The estimate \eqref{L2errorestLeg} seems tighter than this one. \qed
\end{rem}

\section{Error estimates for Gegenbauer-Gauss quadrature}\label{GGquadratrue}
\setcounter{equation}{0} \setcounter{thm}{0} \setcounter{lmm}{0}


\subsection{Preliminaries}  The Gegenbauer-Gauss quadrature remainder
\eqref{gausslob}-\eqref{gausslob2} with the nodes  being  zeros of the Gegenbauer polynomial
$J_n^{\alpha,\alpha}(x)$, takes the form
\begin{equation}\label{errorformu}
E_n^{GG}[u]=\frac{\gamma_n^{\af,\af}}{\pi \ri}\oint_{\mathcal
E_\rho} \frac {Q_n^{\af,\af}(z)} {J_n^{\af,\af}(z)}  u(z)\, dz,\quad \forall\, u\in {\mathcal A}_\rho,
\end{equation}
where $Q_n^{\af,\af}(z)$ is defined as in \eqref{Qexpa}, namely,
\begin{equation}\label{Qexpa0}
Q_n^{\af,\af}(z)=\frac{1}{2\gamma_n^{\alpha,\af}}\int_{-1}^{1}
\frac{J_n^{\alpha,\af}(x)
\omega^{\alpha,\af}(x)}{z-x}\,dx\overset{(\ref{Jcoe3a0s})}=\sum_{j=0}^{\infty}\frac{\sigma_{n,j}^{\af,\af}}{w^{n+j+1}}
\overset{\re{sigmanjodd}}=\sum_{l=0}^{\infty}\frac{\sigma_{n,2l}^{\af,\af}}{w^{n+2l+1}}.
\end{equation}

As already mentioned,  the analysis of quadrature errors (even for the  Chebyshev case) has attracted
 much attention (see e.g.,
\cite{Chawla68,Chawla68SIAM,Basu70,Gautschi83,Dav.R84,GTV90,Hunter1995,Hunter98}). Just to mention that
Chawla and Jain \cite[Theorem 5]{Chawla68} obtained the estimate:
\begin{equation}\label{chwarchbi}
\big|E_n^{CG}[u] \big|\le \frac{2\pi M}{\rho^{2n}-1},\quad \forall\,
u\in {\mathcal A}_\rho, \;\; \forall\, n\ge 1,
\end{equation}
 Hunter \cite{Hunter1995} derived the general bound
\begin{equation}\label{Huntergenbnd}
\big|E_n^{GG}[u]\big|\le
\frac{4\int_{-1}^1(1-x^2)^{\af}dx}{\rho^{2n-2}(\rho^2-1)},\quad
n\ge1,
\end{equation}
and some refined results for $\af=\pm 1/2$ and $\bt=\pm 1/2$ by
expanding $Q_n^{\af,\af}/J_n^{\af,\af}$ into the Laurent series of
$w$ in the disk enclosed by ${\mathcal C}_\rho, $ and manipulating
the series. It is worthwhile to note  that Gautschi and Varga
\cite{Gautschi83} estimated  the Jacobi-Gauss quadrature (with
$J_n^{\af,\bt}$ and $Q_n^{\af,\bt}$ in place of  $J_n^{\af,\af}$ and
$Q_n^{\af,\af}$ in \eqref{errorformu}, respectively)  by
\begin{equation}\label{errorformu2}
\big|E_n^{JG}[u] \big|\le {\pi^{-1}\gamma_n^{\af,\bt}} M L(\mathcal
E_\rho) \max_{z\in {\mathcal E}_\rho}\big|Q_n^{\af,\bt}(z)/J_n^{\af,\bt}(z)\big|,
\end{equation}
and attempted to find the exact maximum value  on the Bernstein ellipse,  which was
feasible  for  $\af=\pm 1/2$ and $\bt=\pm 1/2$ again. Some conjectures and empirical results were explored  in \cite{Gautschi83} for the general Jacobi case.

Using the explicit expression of Legendre polynomials on the Bernstein ellipse (see e.g., \cite[Lemma 12.4.1]{davis1975interpolation}), Kambo \cite{Kambo70} obtained  the bound for the Legendre-Gauss quadrature:
\begin{equation}\label{kambo}
\big|E_n^{LG}[u]\big|\le {\pi^{-1}\gamma_n^{0,0}} M L(\mathcal
E_\rho) \frac { \max_{z\in {\mathcal E_\rho}}|Q_n^{0,0}(z)|}
{\min_{z\in {\mathcal E_\rho}}|J_n^{0,0}(z)|} \le \frac{d_n
M}{\rho^{2n}} \frac{\rho^2+1}{\rho^2-2},\quad \rho>\sqrt {2},
\end{equation}
where $0<d_n\le \pi.$ While this bound is only valid for $\rho>\sqrt {2},$  it holds for all $n,$ when compared with  the asymptotic estimate (with $n\gg 1$) for the Legendre-Gauss quadrature in \cite{Chawla69}.

In what follows, we aim to extend our analysis to estimate  $E_n^{GG}[u]$ in \eqref{errorformu}. The essential tools include
the explicit formula for  the Gegenbauer polynomial $J_n^{\af,\af}(z)$ on ${\mathcal E}_\rho$ derived in our recent paper \cite{XieWangZhao2011}, and the previous argument for estimating $Q_n^{\af,\af}(z).$
Let us recall the  important formula stated in  \cite[Lemma 3.1]{XieWangZhao2011}.
\begin{lemma}\label{lemma1.4}  Let $z=\frac12(w+w^{-1}).$ Then we have
\begin{equation}\label{Gexp}
J_n^{\alpha,\alpha}(z)=A_n^\alpha
\sum_{k=0}^{n}g_k^{\alpha}g_{n-k}^{\alpha}w^{n-2k},\quad n\ge 0,
\;\; \alpha>-1,\;\; \alpha\not=-1/2,
\end{equation}
where
\begin{equation}\label{ajexp0}
g_0^{\alpha}=1,\;\; g_k^{\alpha}
=\frac{\Gamma(k+\alpha+1/2)}{k!\Gamma(\alpha+1/2)},\;\; 1\le k\le n,
\;\;\text{and}\;\;
A_n^\alpha=\frac{\Gamma(2\alpha+1)\Gamma(n+\alpha+1)}{\Gamma(\alpha+1)\Gamma(n+2\alpha+1)}.
\end{equation}
\end{lemma}
\begin{rem}\label{Chebysv} This formula excludes the Chebyshev case. For $\alpha=-1/2,$ we define
\begin{equation}\label{ajexp01}
g_0^{-1/2}=g_n^{-1/2}=1,\;\; g_k^{-1/2}=0,\;\; 1\le k\le n-1,
\;\;\text{and}\;\; A_n^{-1/2}=\frac{\Gamma(n+1/2)}{2\sqrt{\pi}n!},
\end{equation}
since (see e.g., \cite{davis1975interpolation})
\begin{equation}\label{Gexp2}
T_n(z)= \frac 12 (w^n+w^{-n})= \frac1 {2 A_n^{-1/2}} J^{-1/2,-1/2}_n(z).
\end{equation}
  Hence,  we understand that \eqref{Gexp} holds for
$\af=-1/2$ with the constants given by \eqref{ajexp01}.   \qed
\end{rem}

\subsection{Main results}

We adopt two approaches to estimate the quadrature remainder.
The first one is to expand  $Q_n^{\af,\af}/J_n^{\af,\af}$ in Laurent series of $w\in {\mathcal C}_\rho,$ and then
we use an argument as for  Theorem \ref{mainthm} to obtain the tight error bound. However,
this situation is reminiscent to that in Gautschi and Varga \cite{Gautschi83}, that is,  computable bounds can be derived for general $\alpha.$ We highlight that the computational part  (see \eqref{quadrerr}) is  independent of $\rho$ and $u.$

The second approach is based on an important relation between the quadrature remainder and Gegenbauer expansion coefficient (see
\eqref{quadII}).

\vskip 4pt

\def \epsilon {\mu}
The main estimate  resulted from the first approach is stated as follows.
\begin{thm}\label{quadrathmI} For any $u\in {\mathcal A}_\rho$ with $\rho>1,$ we have that for
$\alpha>-1$ and $n\ge 1,$
\begin{equation}\label{quadrerr}
\begin{split}
\big|E_n^{GG}[u]\big|
&\le\gamma_n^{\alpha,\alpha}\Big[\big|\epsilon_{n,0}^{\alpha,\alpha}\big|+\max_{l\ge0}
\big|\epsilon_{n,2l+2}^{\alpha,\alpha}-\epsilon_{n,2l}^{\alpha,\alpha}\big|\frac{1}{\rho^{2}-1}
\Big]\frac{M}{\rho^{2n}},
\end{split}
\end{equation}
where $\{\epsilon_{n,2l}^{\alpha,\alpha}\}_{l\ge0}$ are computed by the recursive formula:
\begin{equation}\label{epsilonnj}
\epsilon_{n,2l}^{\alpha,\alpha}=\frac{1}{g_n^\alpha}\Big(\frac{\sigma_{n,2l}^{\alpha,\alpha}}{A_n^\alpha}-\sum_{k=1}^{\min\{n,l\}}g_k^\alpha
g_{n-k}^{\alpha}\epsilon_{n,2l-2k}^{\alpha,\alpha}\Big),\;\;l\ge1,\quad
\epsilon_{n,0}^{\alpha,\alpha}=\frac{\sigma_{n,0}^{\alpha,\alpha}}{A_n^\alpha
g_n^\alpha}.
\end{equation}
\end{thm}
\begin{proof} A straightforward calculation from \eqref{Qexpa0} (note: $\sigma_{n,2l+1}^{\alpha,\alpha}=0$ for all $l\ge 0$)  and \eqref{Gexp} leads to
\begin{align}\label{ratioexp}
\frac{Q_n^{\alpha,\alpha}(z)}{J_n^{\alpha,\alpha}(z)}=\sum_{l=0}^\infty
\frac{\epsilon_{n,2l}^{\alpha,\alpha}}{w^{2n+2l+1}}
\;\; \text{with}\;\;
\sigma_{n,2l}^{\alpha,\alpha}=A_n^\alpha\sum_{k=0}^{\min\{n,l\}}g_{k}^{\alpha}g_{n-k}^{\alpha}\epsilon_{n,2l-2k}^{\alpha,\alpha},
\end{align}
so solving out  $ \epsilon_{n,2l}^{\alpha,\alpha}$  yields  \eqref{epsilonnj}.

Next, following the same lines  as the derivation of \eqref{expancoepf1}, we infer from \eqref{errorformu} and \eqref{ratioexp} that
 \begin{equation}\label{computbnd}
\begin{split}
\big|E_n^{GG}[u]\big|&
\le\gamma_n^{\alpha,\alpha}\frac{M}{2\pi}\Big|
\sum_{l=0}^{\infty}\epsilon_{n,2l}^{\alpha,\alpha}\oint_{{\mathcal C}_\rho}\frac{1}{w^{2n+2l+1}}\Big(1-\frac{1}{w^2}\Big)\,dw\Big|
\\&\le \gamma_n^{\alpha,\alpha}\frac{M}{2\pi}\Big[\frac{2\pi
\rho}{\rho^{2n+1}}\big|\epsilon_{n,0}^{\alpha,\alpha}\big|+\frac{2\pi
\rho}{\rho^{2n+3}}\sum_{l=0}^{\infty}\big|\epsilon_{n,2l+2}^{\alpha,\alpha}-\epsilon_{n,2l}^{\alpha,\alpha}\big|\frac{1}{\rho^{2l}}\Big]
\\&=\gamma_n^{\alpha,\alpha}\frac{M}{\rho^{2n}}\Big[\big|\epsilon_{n,0}^{\alpha,\alpha}\big|+\frac{1}{\rho^{2}}
\sum_{l=0}^{\infty}\big|\epsilon_{n,2l+2}^{\alpha,\alpha}-\epsilon_{n,2l}^{\alpha,\alpha}\big|\frac{1}{\rho^{2l}}\Big]
\\&\le \gamma_n^{\alpha,\alpha}\frac{M}{\rho^{2n}}\Big[\big|\epsilon_{n,0}^{\alpha,\alpha}\big|+\max_{l\ge0}
\big|\epsilon_{n,2l+2}^{\alpha,\alpha}-\epsilon_{n,2l}^{\alpha,\alpha}\big|\sum_{l=0}^{\infty}\frac{1}{\rho^{2l+2}}
\Big]
\\&= \gamma_n^{\alpha,\alpha}\frac{M}{\rho^{2n}}\Big[\big|\epsilon_{n,0}^{\alpha,\alpha}\big|+\max_{l\ge0}
\big|\epsilon_{n,2l+2}^{\alpha,\alpha}-\epsilon_{n,2l}^{\alpha,\alpha}\big|\frac{1}{\rho^{2}-1}
\Big].
\end{split}
\end{equation}
This completes the proof.
\end{proof}
\begin{rem}\label{bestCheby}
We find from \eqref{epsilonnj} that for  $\alpha=-1/2,$
\begin{equation*}
\epsilon_{n,0}^{-1/2,-1/2}=\frac{2\pi}{\gamma_n^{-1/2,-1/2}},\quad
\big|\epsilon_{n,2l+2}^{-1/2,-1/2}-\epsilon_{n,2l}^{-1/2,-1/2}\big|=\frac{2\pi\delta_{\kappa,0}}{\gamma_n^{-1/2,-1/2}},\;\;
\kappa:=\text{mod}(l+1,n),
\end{equation*}
where $\delta_{\kappa,0}$ is the Kronecker delta. Hence, it follows
from \eqref{computbnd} that
\begin{equation}\label{Chebyquadest}
\big|E_n^{CG}[u]\big|\le \frac{2\pi
M}{\rho^{2n}}\Big[1+\frac{1}{\rho^2}\sum_{j=1}^\infty
\frac{1}{\rho^{2(jn-1)}}\Big]=\frac{2\pi M}{\rho^{2n}-1}, \quad n\ge
1,
\end{equation}
which is the same as  \eqref{chwarchbi} derived in \cite{Chawla68} .\qed
\end{rem}


\begin{rem}\label{bestCheby2nd} We find from \eqref{epsilonnj} that for  $\alpha=1/2,$
\begin{equation}\label{chebmu}
\epsilon_{n,2l}^{1/2,1/2}=\begin{cases}
(-1)^\kappa\dfrac{\pi}{2}\dfrac{1}{\gamma_n^{1/2,1/2}},& \;\;
\text{if}\;\; \kappa:=\text{mod}(l,n+1)=0,1,\\
0,&\;\; \text{otherwise},
\end{cases}
\end{equation}
which implies 
\[
\begin{split}
&\sum_{l=0}^\infty\big|\epsilon_{n,2l+2}^{1/2,1/2}-\epsilon_{n,2l}^{1/2,1/2}\big|\frac{1}{\rho^{2l}}
=\big|\epsilon_{n,2}^{1/2,1/2}-\epsilon_{n,0}^{1/2,1/2}\big|+\sum_{j=1}^\infty\big|\epsilon_{n,2j(n+1)}^{1/2,1/2}\big|\frac{1}{\rho^{2j(n+1)-2}}
\\&\qquad+\sum_{j=1}^\infty\Big(
\big|\epsilon_{n,2j(n+1)+2}^{1/2,1/2}-\epsilon_{n,2j(n+1)}^{1/2,1/2}\big|\frac{1}{\rho^{2j(n+1)}}
+\big|\epsilon_{n,2j(n+1)+2}^{1/2,1/2}\big|\frac{1}{\rho^{2j(n+1)+2}}\Big)
\\&\qquad=\dfrac{\pi}{2}\dfrac{1}{\gamma_n^{1/2,1/2}}\Big(2+(\rho+\rho^{-1})^2\sum_{j=1}^\infty\frac{1}{\rho^{2j(n+1)}}\Big)
=\frac{\pi}{2}\frac{1}{\gamma_n^{1/2,1/2}}\Big(2+\frac{(\rho+\rho^{-1})^2}{\rho^{2n+2}-1}\Big).
\end{split}
\]
Hence, it follows from \eqref{computbnd} that for the Chebyshev-Gauss quadrature of the second kind,
\begin{equation}\label{quadcheb2nd}
  \begin{split}
  \big|E_n^{GG}[u]\big|&\le
  \frac{\pi
  M}{2\rho^{2n}}\bigg(1+\frac{1}{\rho^2}\Big(2+\frac{(\rho+\rho^{-1})^2}{\rho^{2n+2}-1}\Big)\bigg)
  =\frac{\pi
  M(\rho^2+2+\rho^{-2n-4})}{2(\rho^{2n+2}-1)}.
\end{split}
\end{equation}
Note that Hunter \cite[(4.8)]{Hunter1995} obtained the following estimate by a delicate technique:
\begin{equation}\label{Hunterquadcheb2nd}
\big|E_n^{GG}[u]\big|\le \frac{\pi
M(\rho^2+2+\rho^{-2})}{2(\rho^{2n+2}-1)}.
\end{equation}
We see that  \eqref{quadcheb2nd} is sharper.  \qed
\end{rem}

For general $\alpha>-1,$ the derivation of  an explicit
bound  for
\begin{equation}\label{maxboud}
\Theta_n^{\af}:=\max_{l \ge
0} \theta_{n,l}^\af,\quad \theta_{n,l}^\af:=\gamma_n^{\af,\af}\big|\epsilon_{n,2l+2}^{\af,\af}-\epsilon_{n,2l}^{\af,\af}\big|,\; \;\; n\ge 1,
\end{equation}
 seems nontrivial. We have only   empirical results bases on
computation. Some indications are listed as follows.
\begin{itemize}
\item[(i)] Observe from \eqref{chebmu} that for fixed $n$, $\{\theta_{n,l}^{1/2}\}_{l\ge 0}$ are   $(n+1)$-periodic (see Figure \ref{thetagraph} (a)), and the maximum is attained at $l=j(n+1), j=0,1,\cdots$. We compute ample samples of $n,l$ and $\alpha,$ and find very similar ``periodic"  behaviors (see Figure \ref{thetagraph} (b)-(c) for $\alpha=0,1$).

\item[(ii)] Another interesting empirical observation is that for fixed $\alpha,$
the maximum value  $\Theta_n^{\af}$ converges to a constant value, and it
  decreases
as $\alpha$ increases (see Figure \ref{thetagraph} (d)). Note that
for the Legendre case, $\Theta_n^{0}\approx 4.$
\end{itemize}

%

\begin{figure}[!h]
\subfigure[$\theta_{36,l}^{1/2}$]{
\begin{minipage}[t]{0.45\textwidth}
\centering
\rotatebox[origin=cc]{-0}{\includegraphics[width=0.98\textwidth]{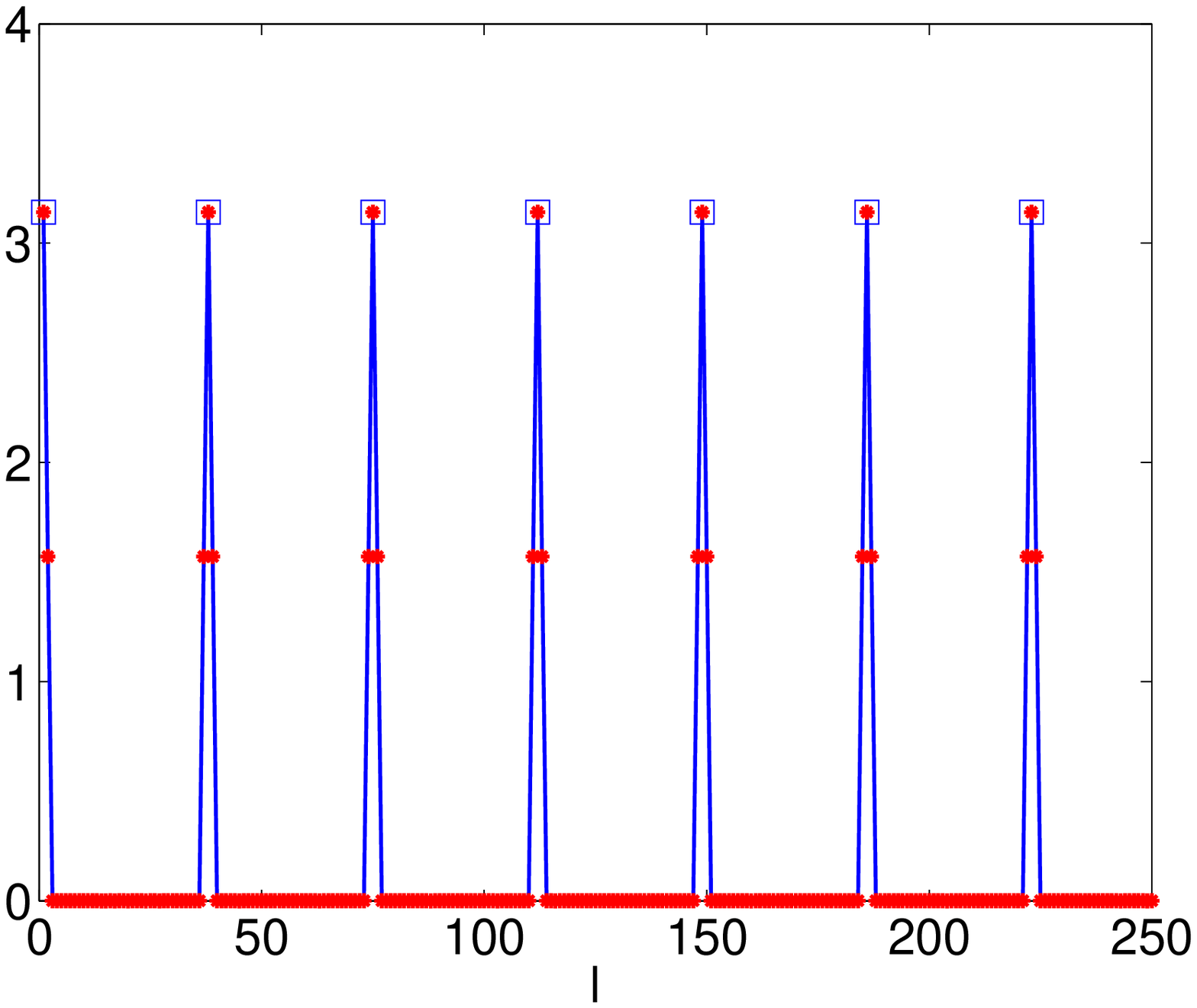}}\label{chebymu}
\end{minipage}}
\subfigure[$\theta_{36,l}^{0}$]{
\begin{minipage}[t]{0.45\textwidth}
\centering
\rotatebox[origin=cc]{-0}{\includegraphics[width=0.98\textwidth]{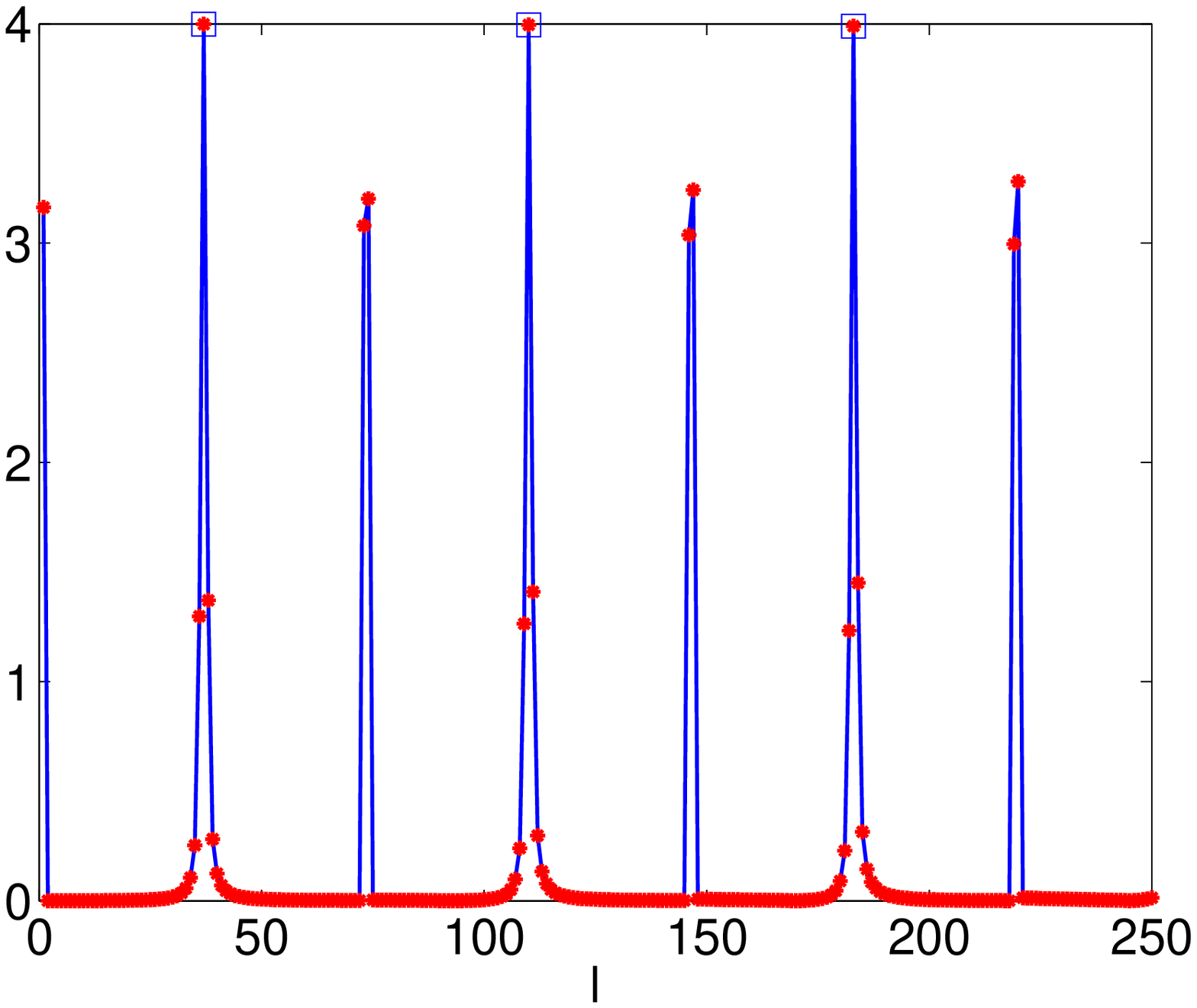}}\label{legmu1}
\end{minipage}}
\subfigure[$\theta_{36,l}^{1}$]{
\begin{minipage}[t]{0.45\textwidth}
\centering
\rotatebox[origin=cc]{-0}{\includegraphics[width=0.98\textwidth]{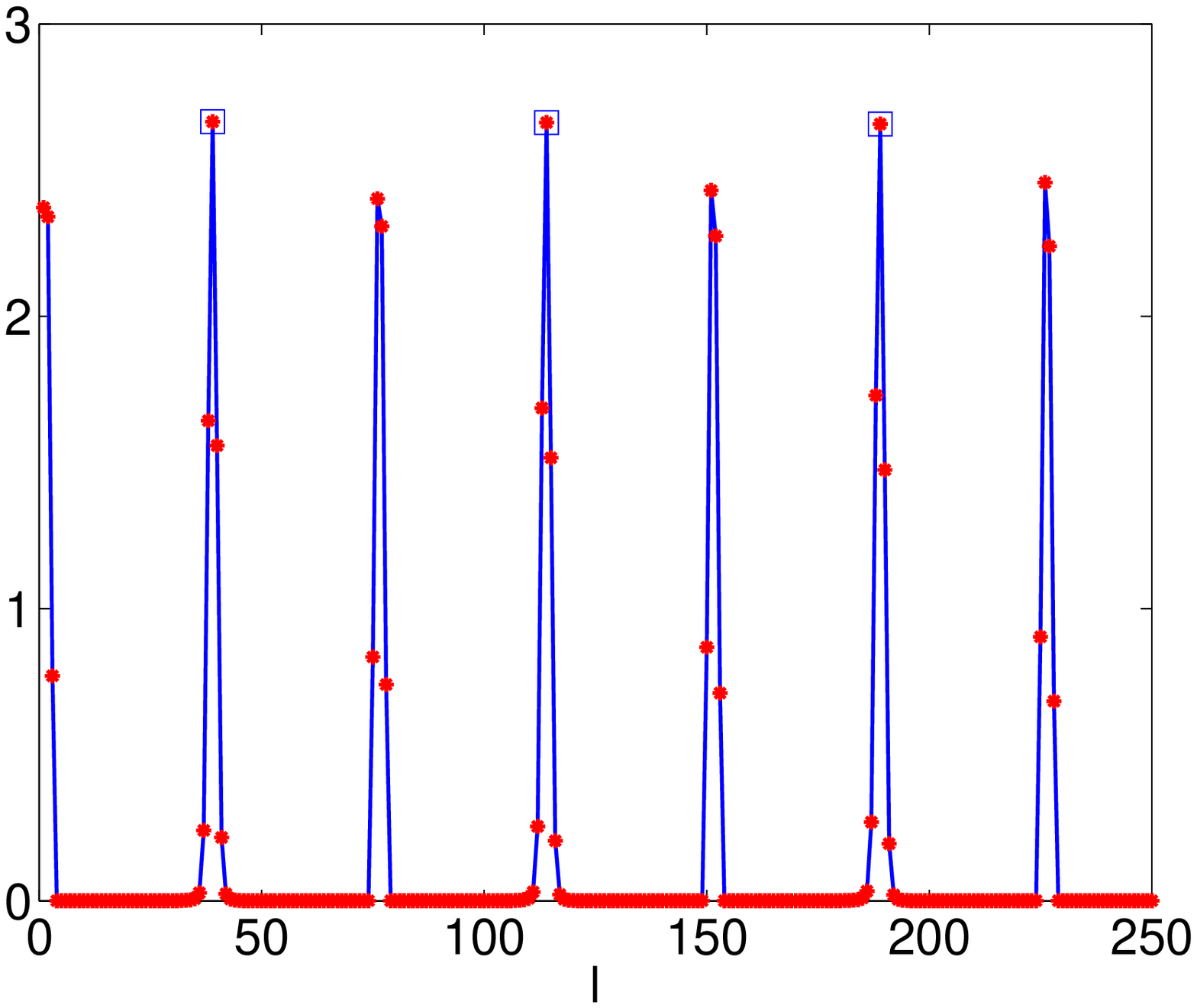}}\label{af1mu}
\end{minipage}}
\subfigure[$\Theta_n^{\af}$ with various $\alpha$]{
\begin{minipage}[t]{0.47\textwidth}
\centering
\rotatebox[origin=cc]{-0}{\includegraphics[width=0.98\textwidth]{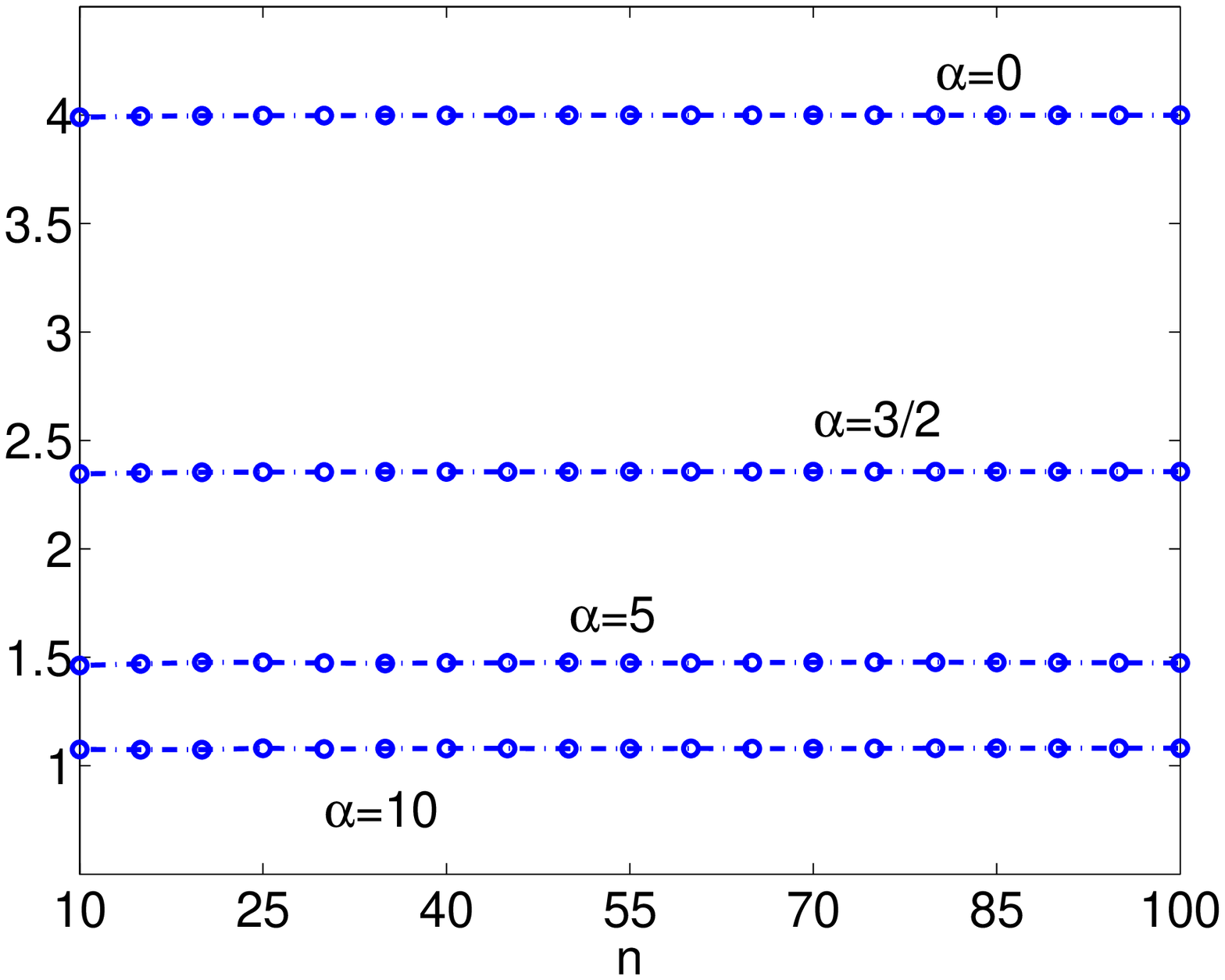}}\label{almaxbnd}
\end{minipage}}
\caption{\small (a)-(c): Profiles of $\theta_{n,l}^{\alpha}$ with $n=36, \alpha=1/2, 0, 1$ and
$0\le l\le 250.$ We mark by  ``{\tiny$\square$}'' the location of the maximum value $\Theta_n^{\af}$ is attached. 
(d): The maximum value $\Theta_n^{\af}$ with $\alpha=0,3/2,5,10$ and $10\le n\le 100,$ where we compute
$\{\theta_{n,l}^{\alpha}\}$ for
$l$ up to $1000.$}\label{thetagraph}
 \end{figure}



Now, we turn to the second approach.  The main result  is summarized below.
\begin{thm}\label{quadrathmIIgegen} For any $u\in {\mathcal A}_\rho$ with $\rho>1,$ and for $\af>-1$ and $\af\not=-1/2,$  we have
\begin{equation}\label{quadraturest2}
\big|E_n^{GG}[u]\big|\le \frac{C_n
M\sqrt{\pi}}{\rho^{2n}}\bigg(\frac{\sqrt{\pi }}{2^{2\af}}+
\frac{\Gamma(\alpha+1)}{\sqrt{\Gamma(2\alpha+2)}}\frac{2\sqrt{2}}{\rho^2-1}
\bigg)
\begin{cases}
(1+\rho^{-2})^{\alpha+1/2},\; &\af>-1/2,\\
(1-\rho^{-2})^{\alpha+1/2},\; & \af<-1/2,
\end{cases}
\end{equation}
and  in particular, for the Legendre case,
\begin{equation}\label{quadrlegenest}
\big|E_n^{LG}[u]\big|\le \frac{C_n
M\pi\sqrt{1+\rho^{-2}}}{\rho^{2n}}\Big(1+\frac{1}{2(\rho^2-1)}\Big),
\end{equation}
where the constant $C_n\approx 1$.
\end{thm}
\begin{proof} We carry out the proof by using the important relation, due to  \eqref{errorformu} and
\eqref{Jcoe3a}:
\begin{equation}\label{quadII}
\begin{split}
\big|E_n^{GG}[u]\big|&\le\frac{\gamma_n^{\alpha,\alpha}}{\min_{z\in\mathcal
E_\rho}|J_n^{\alpha,\alpha}(z)|}\Big|\frac{1}{\pi\ri}\oint_{\mathcal
E_\rho}Q_n^{\alpha,\alpha}(z)u(z)dz\Big|\\
&\overset{\re{Jcoe3a}}=\frac{\gamma_n^{\alpha,\alpha}|\hat u_n^{\alpha,\alpha}|}{\min_{z\in\mathcal
E_\rho}|J_n^{\alpha,\alpha}(z)|}.
\end{split}
\end{equation}
Since the numerator has been estimated in  Theorem \ref{jacobithm} (also see \eqref{asympJacobi2}), it suffices to deal with the denominator.

By \cite[(4.6)]{XieWangZhao2011}, we have
\begin{equation}\label{newconse}
\begin{split}
|J_{n}^{\alpha,\alpha}(z)|&\ge
C_n|A_n^{\alpha}|\frac{n^{\alpha-1/2}\rho^{n}}{|\Gamma(\alpha+1/2)|}
\begin{cases}
(1+\rho^{-2})^{-\alpha-1/2},\quad & {\rm if}\;\; \af>-1/2,\\
(1-\rho^{-2})^{-\alpha-1/2},\quad & {\rm if}\;\; \af<-1/2,
  \end{cases}
  \\&\ge
  C_n\frac{2^{2\af}\rho^{n}}{\sqrt{\pi n}}
\begin{cases}
(1+\rho^{-2})^{-\alpha-1/2},\quad & {\rm if}\;\; \af>-1/2,\\
(1-\rho^{-2})^{-\alpha-1/2},\quad & {\rm if}\;\; \af<-1/2,
  \end{cases}
  \end{split}
\end{equation}
where  $C_n\approx 1.$ Note that in the last step, we dealt with $|A_n^\af|$  as
\[
|A_n^\af|\overset{\re{ajexp0}}=\frac{|\Gamma(2\af+1)|}{\Gamma(\af+1)}\frac{\Gamma(n+\af+1)}{\Gamma(n+2\af+1)}
=\frac{|\Gamma(\af+1/2)|}{2^{-2\af}\sqrt{\pi}}\frac{\Gamma(n+\af+1)}{\Gamma(n+2\af+1)}\ge
C_n\frac{|\Gamma(\af+1/2)|}{2^{-2\af}\sqrt{\pi}n^\af},
\]
where  we used Lemma \ref{lemGammaratio} and the property of Gamma
function (see \cite{Abr.I64}):
\[
\Gamma(z)\Gamma(z+1/2)=2^{1-2z}\sqrt{\pi}\,\Gamma(2z).
\]
Therefore, a combination of \eqref{gammacoef},  \eqref{asympJacobi2} and \eqref{quadII}-\eqref{newconse} leads
to the desired result.

Using  the refined estimate \eqref{legenest4}, yields \eqref{quadrlegenest}.
\end{proof}

\section{Numerical results and comparisons}\label{sect4}

In this section, we present various numerical results to show the tightness of the bounds derived in this paper,
and to compare them  with other existing ones mentioned  in the previous part.

 In the first example, we purposely choose the Chebyshev and Legendre expansions with known expansion coefficients:
\begin{equation} \label{generatingfunc}
u_1(x)=\frac{3}{5-4x}=T_0(x)+\sum_{n=1}^\infty \frac{T_n(x)}{2^{n-1}},
\quad u_2(x)=\frac{2}{\sqrt{5-4x}}=\sum_{n=0}^\infty \frac{L_n(x)}{2^{n}},
\end{equation}
which follow from   generating functions of Chebyshev and Legendre polynomials (cf. \cite{szeg75}).

Note that the function  $u_1$  has a simple pole at $z=5/4,$ so the semi-major axis (cf. \eqref{semiaxis}) should satisfy
\begin{equation*}\label{u1est0}
1<a=(\rho+\rho^{-1})/2<5/4\;\;\Rightarrow\;\; 1<\rho<2.
\end{equation*}
One also verifies that
\begin{equation*}\label{u1est}
M=\max_{z\in {\mathcal E_\rho}}|u_1(z)|=\frac {3\rho}{(2\rho-1)(2-\rho)}.
\end{equation*}
Then the estimate \eqref{chebyest3} reduces to
\begin{equation*}\label{u1est2}
\hat u_n^C=\frac 1 {2^{n-1}}\le \frac {6}{(2\rho-1)(2-\rho)\rho^{n-1}}:=B_n^C(\rho),\quad 1<\rho<2,\;\; n\ge 1.
\end{equation*}
Similarly, for the Legendre expansion of $u_2,$ the result \eqref{legenest4} becomes
\begin{equation*}\label{u2est2}
\hat u_n^{0,0}=\frac 1 {2^{n}}\le
\frac{\sqrt{\pi
n}}{\rho^{n}}\Big(1+\frac{n+2}{2n+3}\frac{1}{\rho^2-1}\Big)\exp\Big(\frac{8n-1}{12n(2n-1)}\Big)\sqrt{\frac {4\rho}{(2\rho-1)(2-\rho)}} :=B_n^L(\rho),
\end{equation*}
for $1<\rho<2 $ and $n\ge 1.$

We take $\rho=1.98,$ and plot the exact coefficients $\hat u_n^C $ and $\hat u_n^{0,0}$, and the bounds $B_n^C$ and $B_n^L$  in  Figure \ref{figchebleg} (a) and (b), respectively. Actually, the bound for the Chebyshev case
(see \eqref{chebmars}) can be considered as one benchmark for illustrating tightness of the upper bound. Indeed,  the result for the Legendre case stated
in Theorem \ref{Th:legexp} seems as sharp as  that for the Chebyshev case.

\begin{figure}[!h]
\subfigure[Chebyshev case]{
\begin{minipage}[t]{0.45\textwidth}
\centering
\rotatebox[origin=cc]{-0}{\includegraphics[width=0.98\textwidth]{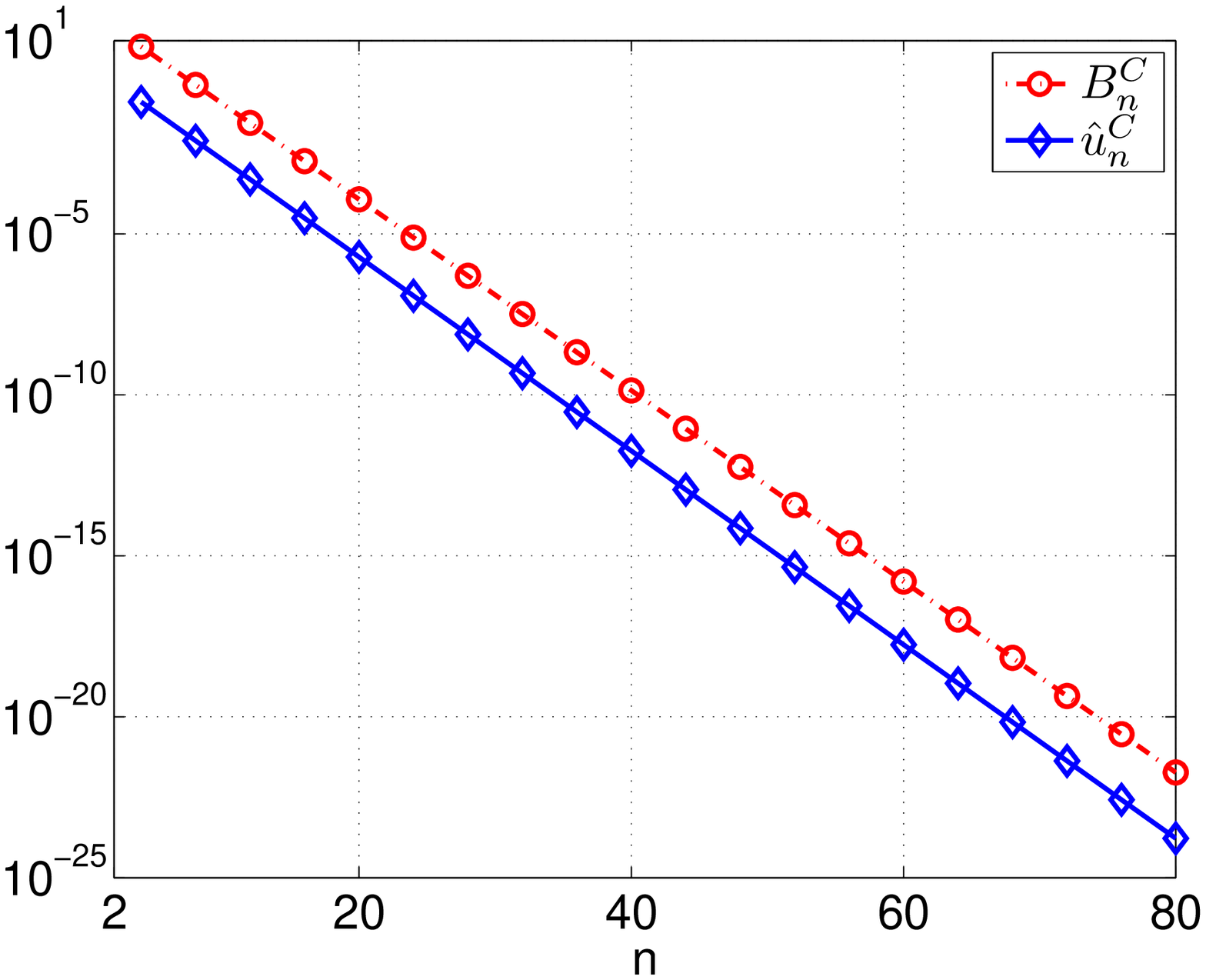}}\label{chebyexcoefbnd}
\end{minipage}}
\subfigure[Legendre case]{
\begin{minipage}[t]{0.45\textwidth}
\centering
\rotatebox[origin=cc]{-0}{\includegraphics[width=0.98\textwidth]{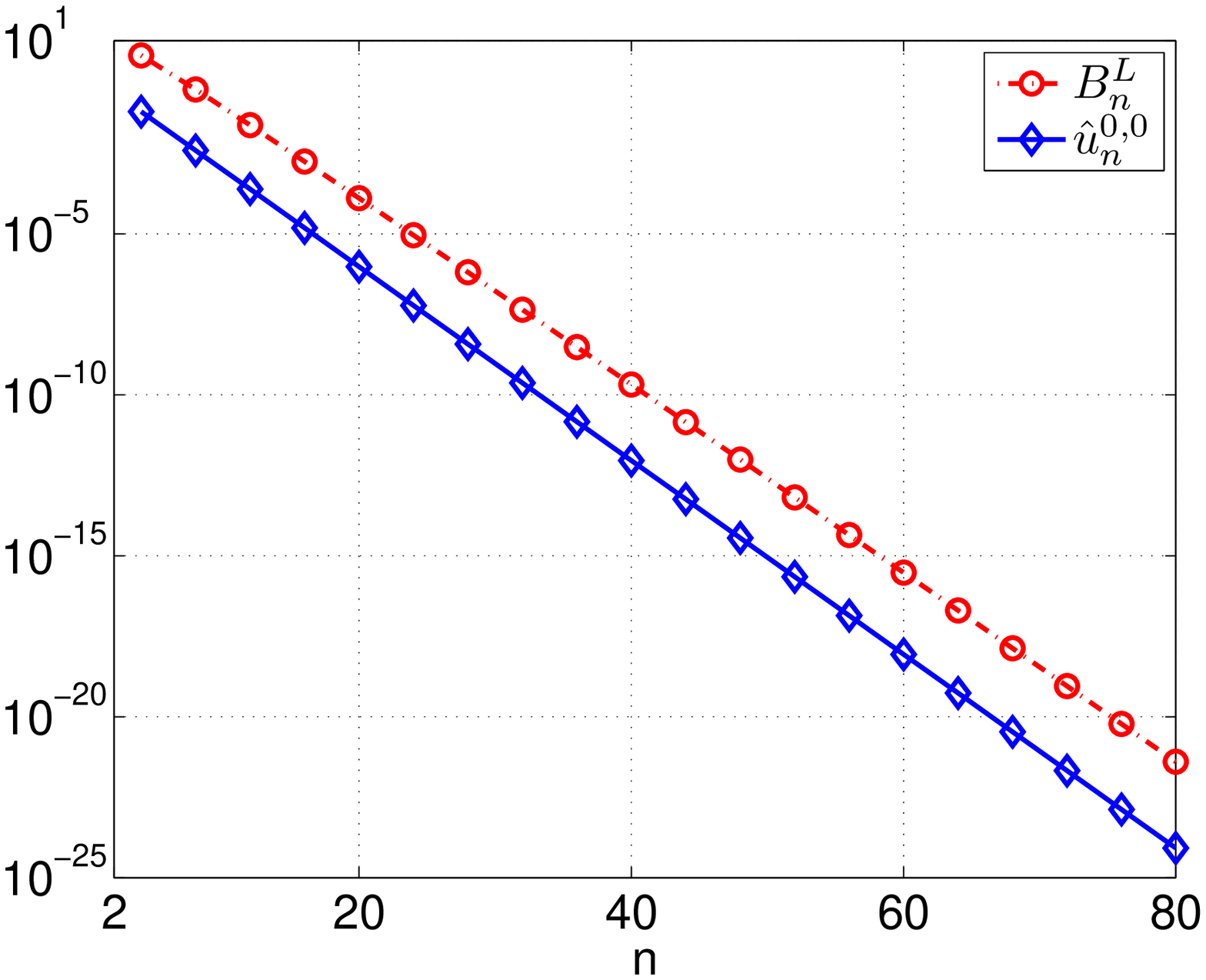}}\label{legenexcoefbnd}
\end{minipage}}
\caption{\small Expansion coefficients of $u_1,u_2$ in (\ref{generatingfunc}) against their error bounds.}\label{figchebleg}
 \end{figure}

\begin{figure}[!h]
\subfigure[$e_n(\rho)$]{
\begin{minipage}[t]{0.47\textwidth}
\centering
\rotatebox[origin=cc]{-0}{\includegraphics[width=1\textwidth]{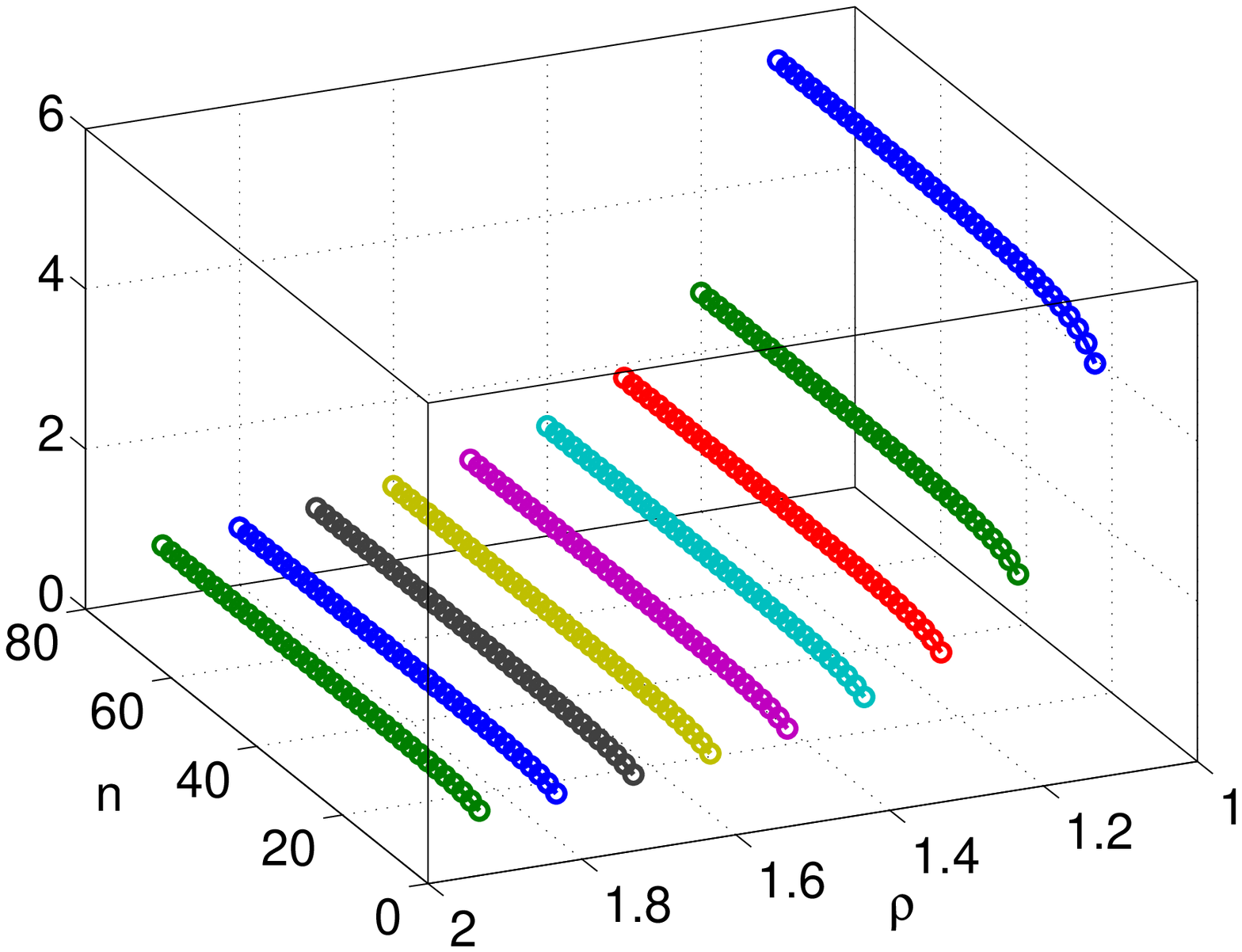}}\label{Legenbndcmp1}
\end{minipage}}
\subfigure[$e_n(\rho)$ with $\rho$ near $1$]{
\begin{minipage}[t]{0.47\textwidth}
\centering
\rotatebox[origin=cc]{-0}{\includegraphics[width=0.95\textwidth]{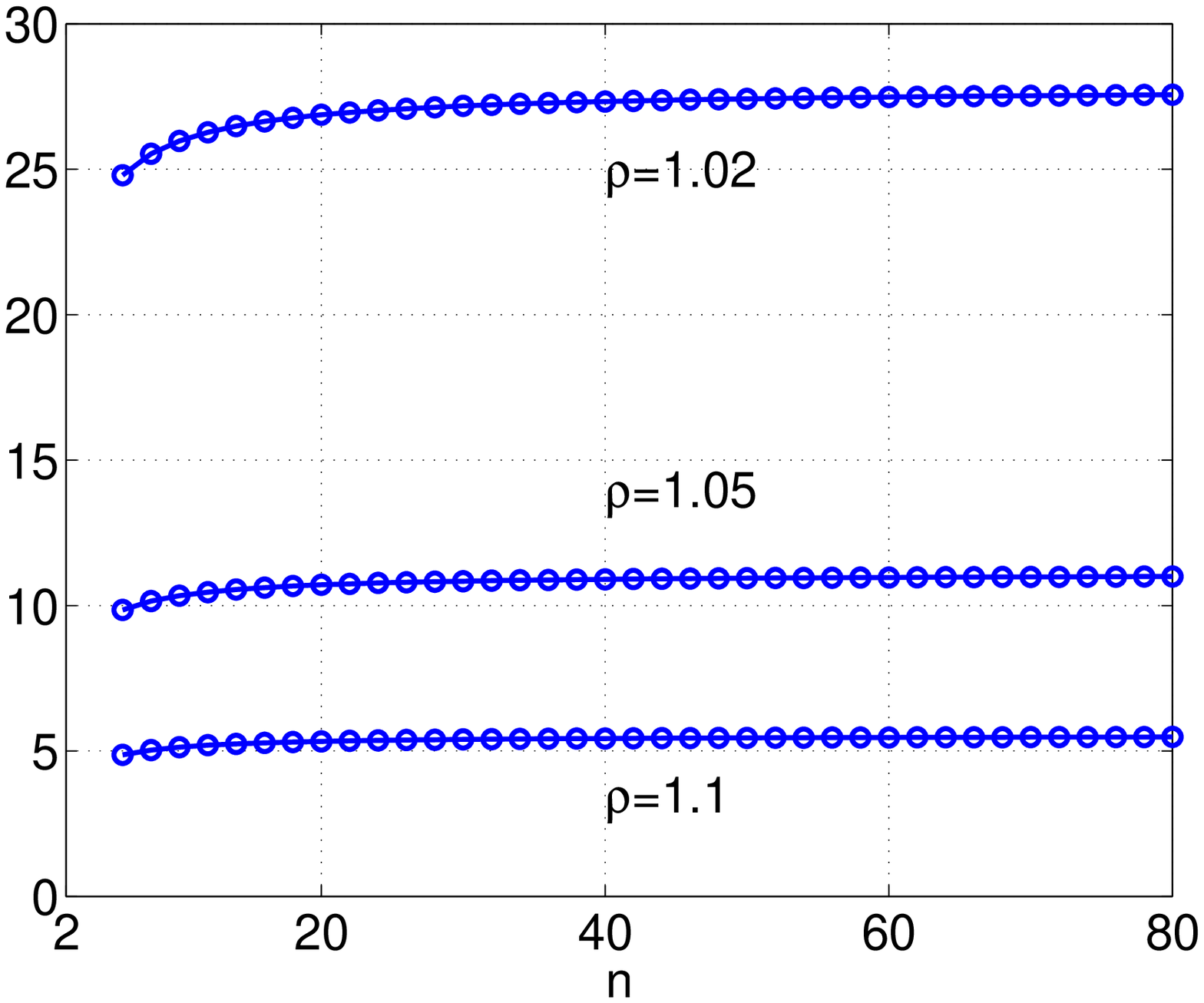}}\label{legentrunbndcmp}
\end{minipage}}
\caption{\small (a): Comparison of error bounds for Legendre
expansions in (\ref{legenest4}) and (\ref{xianglegenbnd}).
 (b): Samples of $e_n(\rho)$ for  $\rho$ close to $1.$ }\label{legcmpbndsa}
 \end{figure}


Next, we compare the bounds for the Legendre expansion coefficients  in Theorem \ref{Th:legexp} 
and \eqref{xianglegenbnd} (obtained by \cite{xiang2012error}). For clarity, we drop  the common part $M\sqrt n /\rho^n,$ and denote the remaining factors in the upper
 bounds by
\begin{equation*}\label{casesdd}
\begin{split}
& b_n(\rho)\overset{(\ref{legenest4})}=\sqrt \pi \Big(1+\frac{n+2}{2n+3}\frac{1}{\rho^2-1}\Big)\exp\Big(\frac{8n-1}{12n(2n-1)}\Big),\quad
   \tilde  b(\rho) \overset{(\ref{xianglegenbnd})} =   2 \Big(1+\frac 1 {\rho^2-1}\Big).
   \end{split}
\end{equation*}
In Figure \ref{legcmpbndsa} (a), we plot the difference
$e_n(\rho):=\tilde  b(\rho)-b_n(\rho)$ for various $\rho$ and $1\le
n\le 80.$ We see that $e_n(\rho)>0$ and  the difference is of
magnitude around $6$, when  $\rho$ is close to $1$. Moreover, for
fixed $\rho,$ the difference is roughly  a constant for slightly
large $n.$ In Figure \ref{legcmpbndsa} (b), we plot some sample
$e_n(\rho)$ for $\rho$ close to 1, we see that our bound is much
sharper.


\begin{figure}[!h]
\subfigure[$\alpha=1,\beta=0$]{
\begin{minipage}[t]{0.47\textwidth}
\centering
\rotatebox[origin=cc]{-0}{\includegraphics[width=1\textwidth]{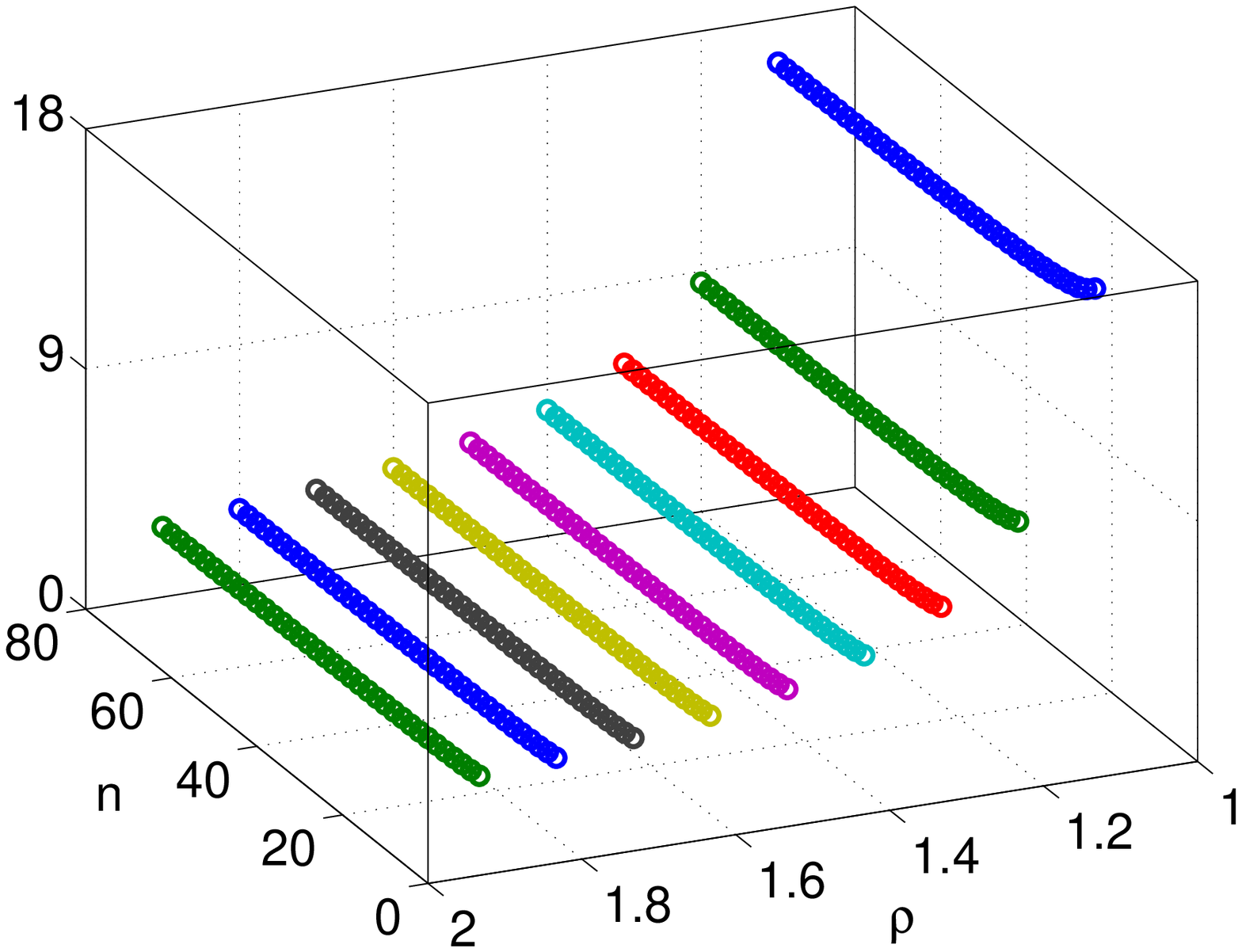}}\label{jacobibndcmp1}
\end{minipage}}
\subfigure[$\alpha=\beta=2$]{
\begin{minipage}[t]{0.47\textwidth}
\centering
\rotatebox[origin=cc]{-0}{\includegraphics[width=1\textwidth]{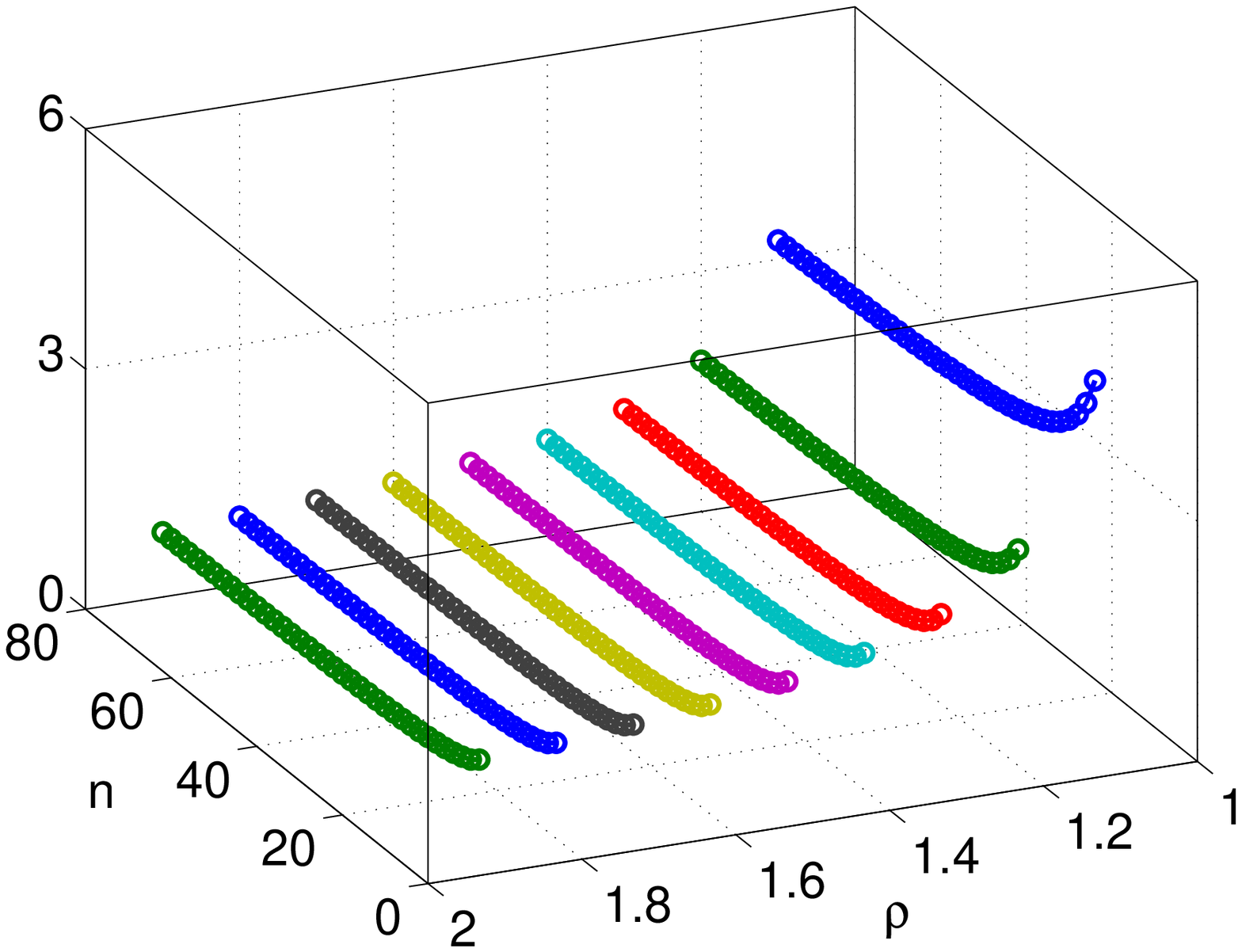}}\label{Gegenbndcmp1}
\end{minipage}}
\caption{\small (a): Comparison of error bounds for Jacobi expansion
with $\af=1,\beta=0$ in \re{xiangres} and \re{legenest10}. (b):
Comparison of error bounds for  Gegenbauer expansion with
$\af=\bt=2$ in \re{xiangres} and \re{gegenest1}.}\label{figcmpbnds3}
 \end{figure}

We next make a similar comparison of bounds for  Jacobi and
Gegenbauer expansions. For example, for $\alpha=1$ and $\beta=0,$ we
extract the factors in \eqref{xiangres} and \eqref{legenest10} by
dropping   $M\sqrt n /\rho^n.$  We plot in Figure
\ref{figcmpbnds3} (a) the difference of two remaining parts (i.e.,
that of  \eqref{xiangres} subtracts that of  \eqref{legenest10}).
Once again, our bound is much tighter.  Likewise, we depict in
Figure \ref{figcmpbnds3} (b) the extracted bounds from
\eqref{xiangres} and \eqref{gegenest1} with $\af=\bt=2.$
The situation is mimic to the Legendre case, where the bounds obtained in this paper are sharper.

Finally, we turn to the comparison of error bounds for the Gegenbauer-Gauss quadrature remainder.
For $\alpha=1/2$, we extract the factors in \eqref{quadcheb2nd} and
\eqref{Hunterquadcheb2nd} by dropping  $M/\rho^{2n}$ as before.  We plot in
Figure \ref{figquadcmp} (a) the difference of two remaining parts
in \eqref{Hunterquadcheb2nd} and in \eqref{quadcheb2nd}). Once again, our bound is much tighter.
Likewise, we depict in Figure \ref{figquadcmp} (b) the extracted
bounds from \eqref{Huntergenbnd} and \eqref{quadrerr} with $\af=2,$ and observe similar behaviors.

\begin{figure}[!h]
\subfigure[$\af=1/2$]{
\begin{minipage}[t]{0.47\textwidth}
\centering
\rotatebox[origin=cc]{-0}{\includegraphics[width=1\textwidth]{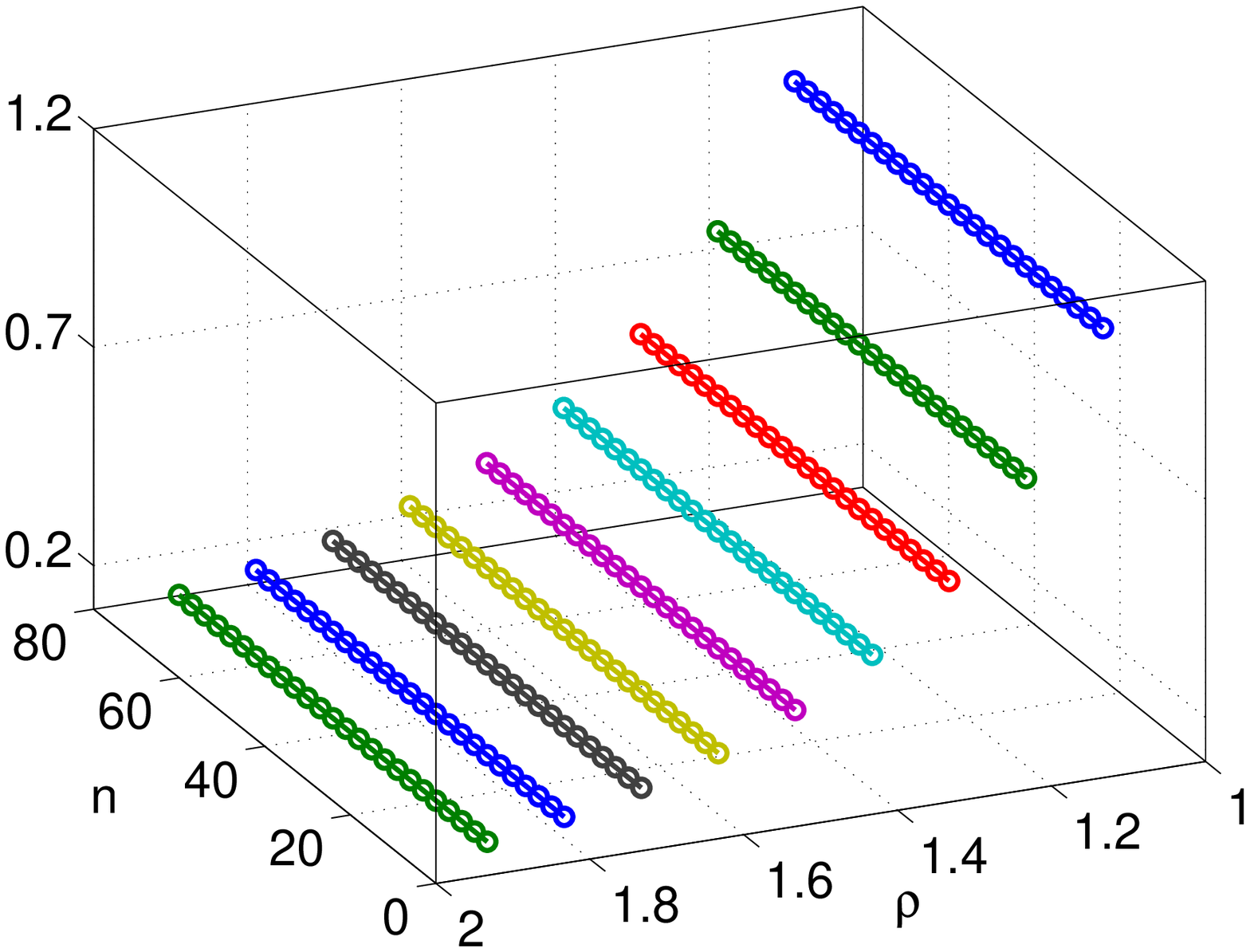}}\label{CG2quadbndcmp}
\end{minipage}}
\subfigure[$\af=2$]{
\begin{minipage}[t]{0.47\textwidth}
\centering
\rotatebox[origin=cc]{-0}{\includegraphics[width=1\textwidth]{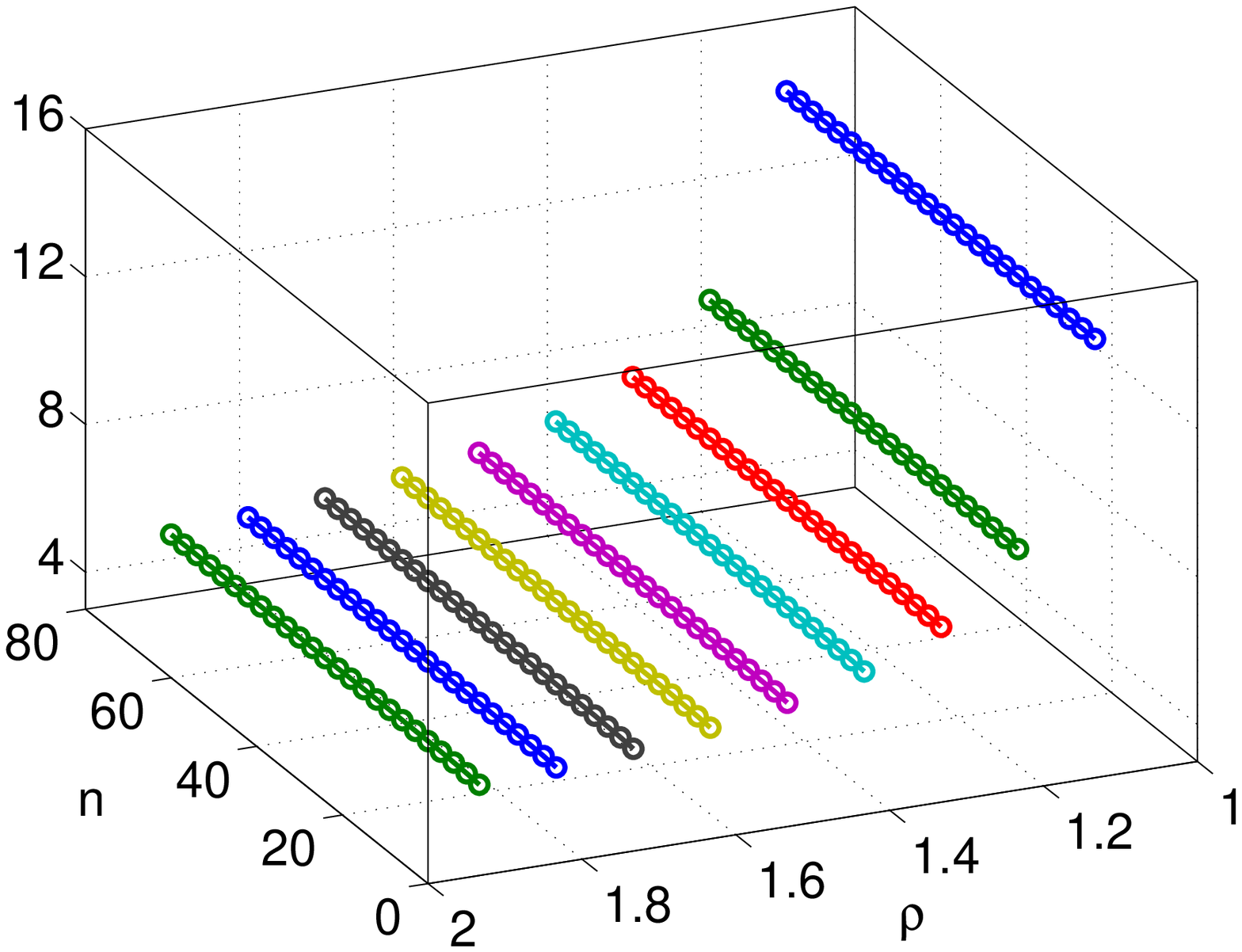}}\label{GG2quadbndcmp}
\end{minipage}}
\caption{\small (a): Comparison of error bounds for the Gegenbauer-Gauss
quadrature with $\af=1/2$ in \re{quadcheb2nd} and
\re{Hunterquadcheb2nd}. (b): Comparison of error bounds for the
Gegenbauer-Gauss quadrature with $\af=2$ in \re{Huntergenbnd} and \re{quadrerr}.}\label{figquadcmp}
 \end{figure}

\vskip 10pt

\noindent\underline{\large \bf Concluding remarks}~
\vskip 6pt
In this paper, we derived various new and sharp error bounds for Jacobi polynomial
expansions  and Gegenbauer-Gauss quadrature of  analytic functions with  analyticity characterized by the Bernstein ellipse.
We  adopted an argument that could recover the best known bounds, and  attempted to make the dependence of the estimates on the parameters explicitly. Both analytic estimates and numerical comparisons with available ones  demonstrated the sharpness of the error bounds.

\vskip 10pt

\begin{appendix}

\renewcommand{\theequation}{A.\arabic{equation}}

\section{Jacobi polynomials}\label{sect:Jacobi}
We collect some  properties of Jacobi polynomials used in the paper.
For  $\alpha, \beta>-1,$ the Jacobi polynomials (see e.g.,
\cite{szeg75}), denoted by $J_n^{\af,\bt}(x), x\in I:=(-1,1),$ are
defined by  the Rodrigues' formula
\begin{equation}\label{Rodrigues}
(1-x)^{\alpha}(1+x)^{\beta}J_n^{\alpha,\beta}(x)=\frac{(-1)^n}{2^n
n!}\frac{d^n}{dx^n}\Big[(1-x)^{\alpha+n}(1+x)^{\beta+n}\Big],\quad
n\ge 0.
\end{equation}
The Jacobi polynomials satisfy
\begin{subequations} \label{mixjacobi}
 \begin{gather}
(1-x)J_n^{\alpha+1,\beta}(x)=\frac{2}{2n+\alpha+\beta+2}\big((n+\alpha+1)J_n^{
\alpha,\beta}(x)-(n+1)J_{n+1}^{\alpha,\beta}(x)\big),\label{mixjacobi1}\\
(1+x)J_n^{\alpha,\beta+1}(x)=\frac{2}{2n+\alpha+\beta+2}\big((n+\beta+1)J_n^{
\alpha,\beta}(x)+(n+1)J_{n+1}^{\alpha,\beta}(x)\big).
\label{mixjacobi2}
 \end{gather}
\end{subequations}
As a  direct consequence of \eqref{mixjacobi}, we have that for any
$k,l\in {\mathbb N}=\{0,1,\cdots\},$
\begin{equation}\label{compexps}
(1-x)^k(1+x)^l J_n^{\af+k,\bt+l}(x)=\sum_{i=n}^{n+k+l}
d_i^{\alpha+k,\beta+l} J_i^{\alpha,\beta}(x),
\end{equation}
where $\{d_i^{\alpha+k,\beta+l}\}_{i=n}^{n+k+l}$ is a unique set of
constants (with $ d_n^{\alpha,\beta}=1$), computed from
\eqref{mixjacobi} recursively. Here, we sketch the proof of \eqref{compexps}. To this end,    let $\{c_j\}$ be a set of generic constants. Using
\eqref{mixjacobi1} and \eqref{mixjacobi2} repeatedly leads to
\begin{equation*}
\begin{split}
&(1-x)^{k}(1+x)^{l}J_n^{\af+k,\bt+ l}(x)\\
&\qquad=(1-x)^{k-1}(1+x)^{l}\big(c_1J_n^{\af+k-1,\bt+l}(x)+c_2J_{n+1}^{\af+k-1,\bt+l}(x)\big)\\
&\qquad=\cdots=(1+x)^{l}\sum_{m=n}^{n+k}c_mJ_m^{\af,\bt+l}(x)=\cdots=\sum_{m=n}^{n+k+l}c_m
J_m^{\af,\bt}(x).
\end{split}
\end{equation*}
This yields \eqref{compexps}.
We point out that for $\af=\bt=0,$ $\{(1-x)^k(1+x)^l J_n^{k,l}\}$ (up to a certain constant factor)  are defined as generalized Jacobi polynomials in \cite{Guo.SW06}.

The following formula, derived from  \cite[Lemma 7.1.1]{GAndrew99}
(also see \cite[Theorem 3.21]{ShenTangWang2011}), was used for the derivation of
\eqref{tempwa0}:
\begin{equation}\label{cocoefng}
\begin{split}
\hat c_j^n:&=\hat c_j^n(\af,\bt,a, b)=\frac 1 {\gamma_n^{\af,\bt}}
\int_{-1}^1 J_{n+j}^{a,b}(x) J_{n}^{\af,\bt}(x) \omega^{\af,\bt}(x)\, dx\\
&=\frac{\Gamma(n+j+a+1)}{\Gamma(n+j+a+b+1)}\frac{(2n+\af+\bt+1)\Gamma(n+\af+\bt+1)}
{\Gamma(n+\af+1)}\\
&\quad \times \sum_{m=0}^{j}\frac{(-1)^m \Gamma(2n+j+m+a+b+1)
\Gamma(n+m+\af+1)}{m!(j-m)!\Gamma(n+m+a+1)\Gamma(2n+m+\af+\bt+2)},
\end{split}
\end{equation}
for $a, b, \af, \bt>-1$ and $n, j\ge 0.$

Let $T_n(x)=\cos(n\,{\rm arccos}(x))$ be the Chebyshev polynomial of
the first kind of degree $n.$ Then  the second-kine  Chebyshev polynomial, denoted by $U_n(x),$ can be expressed
by
\begin{equation}\label{chbytype2}
\begin{split}
U_n(x)=\frac{\sin\big((n+1)\,{\rm
arccos}(x)\big)}{\sqrt{1-x^2}}=\frac{T_{n+1}'(x)}{n+1}
=\sqrt{\frac{\pi}{2}}\frac{J_n^{1/2,1/2}(x)}{\sqrt{\gamma_n^{1/2,1/2}}}.
\end{split}
\end{equation}
The Chebyshev polynomials enjoy  the following important properties:
\begin{subequations} \label{dtn}
 \begin{gather}
 J^{-1/2,-1/2}_n(x)= J^{-1/2,-1/2}_n(1) T_n(x)=\frac{\Gamma(n+1/2)}{\sqrt{\pi}n!}T_n(x)
 ,\label{dtn0}\\
T'_n(x)=2n\underset{k+n\;\text{odd}}{\underset{k=0}{\sum^{n-1}}}
\frac1{c_k}T_k(x), \label{dtn1}
 \end{gather}
\end{subequations}
where $c_0=2$ and $c_k=1$ for $k\ge 1.$

\renewcommand{\theequation}{B.\arabic{equation}}
\section{Proof of Lemma \ref{LM1}}\label{pflemma3.1}

We first show that
\begin{equation}\label{Jcoe3a}
\begin{split}
\hat u_n^{\af,\bt} &=\frac{1}{\pi
\ri}\oint_{\mathcal{E}_\rho}Q_n^{\af,\bt}(z)u(z)\,dz,
\end{split}
\end{equation}
where
\begin{equation}\label{Qexpa}
Q_n^{\af,\bt}(z):=\frac{1}{2\gamma_n^{\alpha,\beta}}\int_{-1}^{1}
\frac{J_n^{\alpha,\beta}(x) \omega^{\alpha,\beta}(x)}{z-x}\,dx,
\end{equation}
and $\gamma_n^{\alpha,\beta}$ is given by \eqref{gammafd}. Recall the Cauchy's integral formula:
\begin{equation}
\begin{split}\label{Cauc}
\frac{d^n}{dx^n}u(x)=\frac{n!}{2\pi
\ri}\oint_{\mathcal{E}_\rho}\frac{u(z)}{(z-x)^{n+1}}\,dz.
\end{split}
\end{equation}
 Using the  Rodrigues' formula \eqref{Rodrigues} and integration by parts leads to
\begin{equation}
\begin{split}\label{Jcoe2}
\hat
u_n^{\af,\bt}&\overset{(\ref{xiangres})}=\frac{1}{\gamma_n^{\alpha,\beta}}\int_{-1}^1u(x)J_n^{\alpha,\beta}(x)\omega^{\alpha,\beta}(x)\,dx
\\&\overset{(\ref{Rodrigues})} =\frac{1}{\gamma_n^{\alpha,\beta}}\frac{(-1)^n}{2^nn!}\int_{-1}^1\omega^{\alpha+n,\beta+n}(x)\frac{d^n}{dx^n}u(x)\,dx\\
&\overset{(\ref{Cauc})}=\frac{1}{\gamma_n^{\alpha,\beta}}\frac{1}{2^nn!}\int_{-1}^1\Big(\frac{n!}{2\pi
\ri}\oint_{\mathcal{E}_\rho}\frac{u(z)}{(z-x)^{n+1}}\,dz\Big)\omega^{\alpha+n,\beta+n}(x)\,dx\\
&=\frac{1}{2^n\gamma_n^{\alpha,\beta}}\frac{1}{2\pi
\ri}\oint_{\mathcal{E}_\rho}\Big(\int_{-1}^1\frac{\omega^{\alpha+n,\beta+n}(x)}{(z-x)^{n+1}}\,dx\Big)u(z)\,dz.
\end{split}
\end{equation}
We find from integration by parts that
\begin{align}\label{eqnadd}
\int_{-1}^1\frac{\omega^{\alpha+n,\beta+n}(x)}{(z-x)^{n+1}}\,dx=
\frac{(-1)^n}{n!}
\int_{-1}^1\frac{1}{z-x}\frac{d^n}{dx^n}\omega^{\alpha+n,\beta+n}(x)\,dx.
\end{align}
Inserting \eqref{eqnadd} into \eqref{Jcoe2},  we derive from  the
Rodrigues' formula \eqref{Rodrigues} that
\begin{equation*}
\begin{split}\label{Jcoe3}
\hat u_n^{\af,\bt} &=\frac{1}{2\pi
\ri}\frac{1}{\gamma_n^{\alpha,\beta}}\oint_{\mathcal{E}_\rho}\Big(\int_{-1}^1\frac{\omega^{\alpha,\beta}(x)J_n^{\alpha,\beta}(x)}{z-x}\,dx\Big)u(z)\,dz
=\frac{1}{\pi \ri}\oint_{\mathcal{E}_\rho}Q_n^{\af,\bt}(z)u(z)\,dz,
\end{split}
\end{equation*}
where $Q_n^{\af,\bt}(z)$ is given in \eqref{Qexpa}.

Since $z=(w+w^{-1})/2,$   we have from the generating function of the Chebyshev polynomial of the second-kind
(cf. \cite{Abr.I64}) that
\begin{equation}\label{Useries}
\frac 1 {z-x}=\frac 2 w \frac{1}{w^{-2}-2xw^{-1}+1}=\frac 2 w
\sum_{k=0}^\infty \frac{U_k(x)}{w^k}.
\end{equation}
Inserting it into  \eqref{Qexpa}, we find from the  orthogonality of the Jacobi
polynomials (cf. \eqref{jacobi_orth}) that
\begin{equation}\label{Jcoe3a0s}
\begin{split}
Q_n^{\af,\bt}(z) &=\frac{1}{\gamma_n^{\alpha,\beta}}\sum_{k=0}^\infty \frac 1 {w^{k+1}}
\int_{-1}^{1} U_k(x) J_n^{\alpha,\beta}(x) \omega^{\alpha,\beta}(x)\,dx\\
&=\frac{1}{\gamma_n^{\alpha,\beta}}\sum_{k=n}^\infty \frac 1 {w^{k+1}}
\int_{-1}^{1} U_k(x) J_n^{\alpha,\beta}(x) \omega^{\alpha,\beta}(x)\,dx \\
&=\frac{1}{\gamma_n^{\alpha,\beta}}\sum_{j=0}^\infty \frac 1 {w^{n+j+1}}
\int_{-1}^{1} U_{n+j}(x) J_n^{\alpha,\beta}(x) \omega^{\alpha,\beta}(x)\,dx
=\sum_{j=0}^{\infty}\frac{\sigma_{n,j}^{\af,\bt}}{w^{n+j+1}},
\end{split}
\end{equation}
where we defined
\begin{equation*}
\sigma_{n,j}^{\af,\bt}= \frac{1}{\gamma_n^{\af,\bt}}
 \int_{-1}^1U_{n+j}(x) J_n^{\alpha,\beta}(x)\omega^{\alpha,\beta}(x)
 \,dx,
 \end{equation*}
Substituting the last identity of \eqref{Jcoe3a0s} into \eqref{Jcoe3a} leads to the desired formula \eqref{expancoef}.

\end{appendix}



\begin{thebibliography}{}

\end{thebibliography}


\begin{thebibliography}{10}

\bibitem{Abr.I64}
M.~Abramovitz and I.A. Stegun.
\newblock {\em {Handbook of Mathematical Functions}}.
\newblock Dover, New York, 1972.

\bibitem{GAndrew99}
G.E. Andrews, R.~Askey, and R.~Roy.
\newblock {\em Special Functions}, volume~71 of {\em Encyclopedia of
  Mathematics and its Applications}.
\newblock Cambridge University Press, Cambridge, 1999.

\bibitem{Basu70}
N.K. Basu.
\newblock Error estimates for a {C}hebyshev quadrature method.
\newblock {\em Math. Comp.}, 24:863--867, 1970.

\bibitem{BernardiMaday97}
C.~Bernardi and Y.~Maday.
\newblock {Spectral Methods}.
\newblock In {\em {Handbook of Numerical Analysis, {V}ol. {V}}}, Handb. Numer.
  Anal., V, pages 209--485. North-Holland, Amsterdam, 1997.

\bibitem{Bernstain}
S.N. Bernstein.
\newblock Sur l'ordre de la meilleure approximation des fonctions continues par
  des polynomes de degre donne.
\newblock {\em M´emoires publi´es par la class des sci. Acad. de Belgique},
  2(4):1--103, 1912.

\bibitem{Boyd94}
J.P. Boyd.
\newblock The rate of convergence of {F}ourier coefficients for entire
  functions of infinite order with application to the {W}eideman-{C}loot
  sinh-mapping for pseudospectral computations on an infinite interval.
\newblock {\em J. Comput. Phys.}, 110(2):360--372, 1994.

\bibitem{Boyd01}
J.P. Boyd.
\newblock {\em {C}hebyshev and {F}ourier {S}pectral {M}ethods}.
\newblock Dover Publications Inc., 2001.

\bibitem{CHQZ06}
C.~Canuto, M.Y. Hussaini, A.~Quarteroni, and T.A. Zang.
\newblock {\em {Spectral Methods: Fundamentals in Single Domains}}.
\newblock Springer, Berlin, 2006.

\bibitem{Chawla69}
M.M. Chawla.
\newblock Asymptotic estimates for the error of the {G}auss-{L}egendre
  quadrature formula.
\newblock {\em Comput. J.}, 11:339--340, 1968/1969.

\bibitem{Chawla68SIAM}
M.M. Chawla.
\newblock On {D}avis's method for the estimation of errors of
  {G}auss-{C}hebyshev quadratures.
\newblock {\em SIAM J. Numer. Anal.}, 6:108--117, 1969.

\bibitem{Chawla68}
M.M. Chawla and M.K. Jain.
\newblock Error estimates for {G}auss quadrature formulas for analytic
  functions.
\newblock {\em Math. Comp.}, 22:82--90, 1968.

\bibitem{davis1975interpolation}
P.J. Davis.
\newblock {\em {Interpolation and Approximation}}.
\newblock Dover Publications, Inc, New York, 1975.

\bibitem{Dav.R84}
P.J. Davis and P.~Rabinowitz.
\newblock {\em Methods of Numerical Integration}.
\newblock Computer Science and Applied Mathematics. Academic Press Inc.,
  Orlando, FL, second edition, 1984.

\bibitem{Fornberg96}
B.~Fornberg.
\newblock {\em A Practical Guide to Pseudospectral Methods}, volume~1 of {\em
  Cambridge Monographs on Applied and Computational Mathematics}.
\newblock Cambridge University Press, Cambridge, 1996.

\bibitem{Funa92}
D.~Funaro.
\newblock {\em Polynomial Approxiamtions of Differential Equations}.
\newblock Springer-Verlag, 1992.

\bibitem{GTV90}
W.~Gautschi, E.~Tychopoulos, and R.S. Varga.
\newblock A note on the contour integral representation of the remainder term
  for a {G}auss-{C}hebyshev quadrature rule.
\newblock {\em SIAM J. Numer. Anal.}, 27(1):219--224, 1990.

\bibitem{Gautschi83}
W.~Gautschi and R.S. Varga.
\newblock Error bounds for {G}aussian quadrature of analytic functions.
\newblock {\em SIAM J. Numer. Anal.}, 20(6):1170--1186, 1983.

\bibitem{gottlieb1977numerical}
D.~Gottlieb and S.A. Orszag.
\newblock {\em {Numerical Analysis of Spectral Methods: Theory and
  Applications}}.
\newblock Society for Industrial Mathematics, 1977.

\bibitem{Got.S97}
D.~Gottlieb and C.W. Shu.
\newblock On the {G}ibbs phenomenon and its resolution.
\newblock {\em SIAM Rev.}, 39(4):644--668, 1997.

\bibitem{gottlieb1992gibbs}
D.~Gottlieb, C.W. Shu, A.~Solomonoff, and H.~Vandeven.
\newblock {On the Gibbs phenomenon I: recovering exponential accuracy from the
  Fourier partial sum of a nonperiodic analytic function}.
\newblock {\em Journal of Computational and Applied Mathematics},
  43(1-2):81--98, 1992.

\bibitem{Guo98}
B.Y. Guo.
\newblock {\em Spectral Methods and Their Applications}.
\newblock World Scientific Publishing Co. Inc., River Edge, NJ, 1998.

\bibitem{Guo.SW06}
B.Y.~Guo, J.~Shen, and L.L. Wang.
\newblock Optimal spectral-{G}alerkin methods using generalized {J}acobi
  polynomials.
\newblock {\em J. Sci. Comput.}, 27(1-3):305--322, 2006.

\bibitem{Hale.Tr08}
N.~Hale and L.N.~Trefethen.
\newblock New quadrarture formulas from conformal maps.
\newblock {\em SIAM J. Numer. Anal.}, 46(2):930--948, 2008.


\bibitem{GotHes07}
J.S. Hesthaven, S.~Gottlieb, and D.~Gottlieb.
\newblock {\em Spectral Methods for Time-Dependent Problems}, volume~21 of {\em
  Cambridge Monographs on Applied and Computational Mathematics}.
\newblock Cambridge University Press, Cambridge, 2007.

\bibitem{Hunter1995}
D.B. Hunter.
\newblock Some error expansions for {G}aussian quadrature.
\newblock {\em BIT}, 35(1):64--82, 1995.

\bibitem{Hunter98}
D.B. Hunter and G.~Nikolov.
\newblock Gaussian quadrature of {C}hebyshev polynomials.
\newblock {\em J. Comput. Appl. Math.}, 94(2):123--131, 1998.

\bibitem{Kambo70}
N.S. Kambo.
\newblock Error of the {N}ewton-{C}otes and {G}auss-{L}egendre quadrature
  formulas.
\newblock {\em Math. Comp.}, 24:261--269, 1970.

\bibitem{lorentz}
G.G. Lorentz.
\newblock {\em Approximation of Functions}.
\newblock AMS Chelsea Publishing Company, 1966.

\bibitem{Mason03}
J.C. Mason and D.C. Handscomb.
\newblock {\em Chebyshev Polynomials}.
\newblock Chapman \& Hall/CRC, Boca Raton, FL, 2003.

\bibitem{ReddyWeideman2005}
S.C. Reddy and J.A.C. Weideman.
\newblock The accuracy of the {C}hebyshev differencing method for analytic
  functions.
\newblock {\em SIAM J. Numer. Anal.}, 42(5):2176--2187 (electronic), 2005.

\bibitem{rivlin1990chebyshev}
T.J. Rivlin and V.~Kalashnikov.
\newblock {\em {Chebyshev Polynomials: From Approximation Theory to Algebra and
  Number Theory}}.
\newblock Wiley New York, 1990.

\bibitem{ShenTangWang2011}
J.~Shen, T.~Tang, and L.L. Wang.
\newblock {\em {Spectral Methods : Algorithms, Analysis and Applications}},
  volume~41 of {\em Series in Computational Mathematics}.
\newblock Springer, 2011.

\bibitem{szeg75}
G.~Szeg\"o.
\newblock {\em Orthogonal Polynomials (fourth edition)}.
\newblock AMS Coll. Publ., 1975.

\bibitem{tadmor1986exponential}
E.~Tadmor.
\newblock {The exponential accuracy of Fourier and Chebyshev differencing
  methods}.
\newblock {\em SIAM Journal on Numerical Analysis}, 23(1):1--10, 1986.

\bibitem{Tref00}
L.N. Trefethen.
\newblock {\em Spectral Methods in {MATLAB}}.
\newblock Software, Environments, and Tools. Society for Industrial and Applied
  Mathematics (SIAM), Philadelphia, PA, 2000.

\bibitem{TrefSIAMRev08}
L.N. Trefethen.
\newblock Is {G}auss quadrature better than {C}lenshaw-{C}urtis?
\newblock {\em SIAM Rev.}, 50(1):67--87, 2008.

\bibitem{Xiang2011}
H.Y. Wang and S.H. Xiang.
\newblock On the convergence rates of {L}egendre approximation.
\newblock {\em Math. Comp.}, 81(278):861--877, 2012.

\bibitem{xiang2012error}
S.H. Xiang.
\newblock On error bounds for orthogonal polynomial expansions and {G}auss-type
  quadrature.
\newblock {\em SIAM J. Numer. Anal.}, 50(3):1240--1263, 2012.

\bibitem{XieWangZhao2011}
Z.Q. Xie, L.L. Wang, and X.D. Zhao.
\newblock On exponential convergence of {G}egenbauer interpolation and spectral
  differentiation.
\newblock {\em Math. Comp.}, available  online since August 21, 2012.

\bibitem{ZhangZM04}
Z.M.~Zhang.
\newblock Superconvergence of spectral collocation and {$p$}-version methods in
  one dimensional problems.
\newblock {\em Math. Comp.}, 74(252):1621--1636 (electronic), 2005.

\bibitem{ZhangZM08}
Z.M.~Zhang.
\newblock Superconvergence of a {C}hebyshev spectral collocation method.
\newblock {\em J. Sci. Comput.}, 34(3):237--246, 2008.

\bibitem{ZhangInp2012}
Z.M.~Zhang.
\newblock Superconvergence points of spectral interpolation.
\newblock {\em arXiv:1204.5813}, 2012.

\end{thebibliography}

\end{document}